\documentclass[10pt,a4paper,reqno,oneside]{amsart}

\usepackage{amsrefs}

\usepackage{array}
\usepackage{url}
\usepackage{multicol} 
\usepackage[linktocpage = true]{hyperref}

\hypersetup{
 colorlinks = true,
 citecolor = blue,
}
\usepackage{color}
\definecolor{red}{rgb}{1,0,0}
\definecolor{green}{rgb}{0,1,0}
\definecolor{blue}{rgb}{0,0,1}
\definecolor{refkey}{gray}{.625}
\definecolor{labelkey}{gray}{.625}
\usepackage[a4paper]{geometry}
\usepackage{amsmath,amssymb}
\usepackage{amssymb}
\usepackage{amsmath}
\usepackage{amsthm}
\usepackage{enumerate}
\usepackage{graphicx}

\usepackage{tikz-cd}
\allowdisplaybreaks

\makeatletter

\theoremstyle{plain}
\newtheorem{thm}{\protect\theoremname}[section]
\newtheorem{prop}[thm]{\protect\propositionname}

\newtheorem{lem}[thm]{\protect\lemmaname}

\theoremstyle{definition}
 
\newtheorem{defn}[thm]{\protect\definitionname} 
\newtheorem{remark}[thm]{\protect\remarkname} 
\newtheorem{notation}[thm]{\protect\notationname}

\AtBeginDocument{
	
}

\makeatother

  \providecommand{\corollaryname}{Corollary}
  \providecommand{\examplename}{Example}
  \providecommand{\lemmaname}{Lemma}
  \providecommand{\propositionname}{Proposition}
  \providecommand{\theoremname}{Theorem}
  \providecommand{\definitionname}{Definition}
  \providecommand{\remarkname}{Remark}
  \providecommand{\notationname}{Notation}

  \makeatletter
  \newcommand{\be}{%
  \begingroup
  \eqnarray%
   \@ifstar{\nonumber}{}%
  }

\makeatother

\begin{document}

\title[Drinfeld Module and Weil pairing over Dedekind domain of class number two]{Drinfeld Module and Weil pairing over Dedekind domain of class number two}
\thanks{ Research partially supported by NSFC grants 
	12071247/12101616(Hu),  Guangdong Basic and Applied Basic Research Foundation No. 2021A1515110654(Huang) and
 the Basic and Applied Basic Research of Guangzhou Basic Research Program No. 
202201010234(Huang) .}
\thanks{$^*$ The corresponding author.}

\author{Chuangqiang Hu}
\address{Beijing Institute of Mathematical Sciences and Applications, Beijing, 101408, China}

\email{\href{huchq@bimsa.cn}{huchq@bimsa.cn}}
 
\author{Xiao-Min Huang$^*$}
\address{School of Mathematics and statistics, 
Guangdong University of Technology, Guangzhou, Guangdong 510006, 
Guangdong, China}
\email{\href{mahuangxm@gdut.edu.cn}{mahuangxm@gdut.edu.cn}}

\allowdisplaybreaks

\begin{abstract}
The primary objective of this paper is to derive explicit formulas for
rank one and rank two Drinfeld modules over a specific domain denoted by
$\mathcal{A}$. This domain corresponds to the projective line associated
with an infinite place of degree two. To achieve the goals, we construct
a pair of  standard rank one Drinfeld modules whose coefficients are in the Hilbert
class field of $\mathcal{A}$. We demonstrate that the period lattice of
the exponential functions corresponding to both modules behaves similarly
to the period lattice of the Carlitz module, the standard rank one Drinfeld
module defined over rational function fields. Moreover, we employ Anderson's
$t$-motive to obtain the complete family of rank two Drinfeld modules.
This family is parameterized by the invariant
$J = \lambda ^{q^{2}+1}$ which effectively serves as the counterpart of
the $j$-invariant for elliptic curves. Building upon the concepts introduced
by van~der~Heiden, particularly with regard to rank two Drinfeld modules,
we are able to reformulate the Weil pairing of Drinfeld modules of any
rank using a specialized polynomial in multiple variables known as the
Weil operator. As an illustrative example, we provide a detailed examination
of a more explicit formula for the Weil pairing and the Weil operator of
rank two Drinfeld modules over the domain $\mathcal{A}$.
\end{abstract}
\maketitle{}



\tableofcontents

\section{Introduction}
\label{sec1}

Drinfeld modules are function field analogues of elliptic curves that were
studied by Drinfeld \cite{DVG74} in the 70's for proving the Langlands
conjectures for $\mathrm{GL}_{2}$. Drinfeld module encompasses a rich arithmetic
theory, and indeed a number of special cases were already studied by Carlitz
and Wade (see \cite{CL35,Wade46} for instance) prior to Drinfeld's general
definition. Works of many authors (such as Drinfeld, Hayes, Goss, Gekeler,
Anderson, Brownawell, Thakur, Papanikolas and Cornelissen, to name a few)
established many striking similarities, and yet a few astounding differences
between elliptic curves and Drinfeld modules (particularly in rank 2);
the reader may consult \cite{Goss96}, \cite{P23},
\cite{PB21} and \cite{Tha04} for a more elaborate discussion. The comparison
of similarities and differences in the statements (and proofs) between
corresponding results in the classical and Drinfeld theories provides valuable
insights on some common themes in number theory.

\subsection{Explicit representations of Drinfeld modules}
\label{sec1.1}

Many explicit computations of Drinfeld modules are based on
polynomial ring, since polynomial ring is free generated,
and there is no obstruction for constructing Drinfeld modules. Precisely,
for a polynomial ring $ A = \mathbb{F}_{q}[T] $, a rank $r$ Drinfeld module
$ \phi $ is represented as a twist polynomial
%
\begin{equation}
\label{Eq:phioverpoly}
\phi _{T} = \iota (T) + \sum _{i=1}^{r} g_{i} \tau ^{i},
\end{equation}
where $\tau $ is the $q$-th Frobenius endomorphism. Among these Drinfeld
modules, the rank one module of the form
\begin{equation*}
C_{T} = T + \tau
\end{equation*}
is called Carlitz module. One of the motivations for studying the Carlitz
module is to explore explicit class field theory for the rational function
field as in \cite{Car38}.

There are few explicit instances of Drinfeld modules defined over non-polynomial
rings. In \cite{DH94}, Dummit and Hayes computed the explicit form of the
sgn-normalized rank-one Drinfeld modules associated with all elliptic curves
over $\mathbb F_{q}$ for $4\leqslant q\leqslant 13$, where the infinite
place is ramified. Extended this work, Bae and Kang \cite{BK95} investigated
the rank one Drinfeld modules on elliptic curves and hyper-elliptic curves
in the case that the infinite place is inert. In \cite{GP18}, Green and
Papanikolas made a detailed account of sign-normalized rank one Drinfeld
modules over the coordinate ring of an elliptic curve, which provide a
parallel theory to the Carlitz module.

In the present paper, we focus on the domain
$ \mathcal{A}:= \mathbb{F}_{q} [x, y] $, where $ x$ and
$ y $ verify the relation:
\begin{equation*}
y^{2} - \mathrm{Tr}(\zeta ) xy + \zeta ^{q+1} x^{2} - x = 0,
\end{equation*}
and $ \zeta $ is a generator of $ \mathbb{F}_{q^{2}} $. In
fact, the fraction field of $ \mathcal{A}$ is equal to the rational function
field $\mathbb F_{q}(t)$ with $t=\frac{y}{x}$. From a geometric perspective,,
it corresponds to the affine domain of a degree-two-point truncated projective
line. We describe clearly the rank one and rank two Drinfeld modules over
$ \mathcal{A}$. Considering the leading terms, Drinfeld modules are split
into two types, say $ \zeta $-type and $ \zeta ^{q} $-type. According to
the Hayes module theory, the two types are in one-to-one correspondence
with the ideal classes of $ \mathcal{A}$. Among these, we obtain two standard
rank one Drinfeld modules corresponding to the ideal classes of
$ \mathcal{A}$, denoted by $ \Psi $ and $ \Psi ^{\sigma }$. By abuse of terminology we set $ T := \frac{1}{t -\zeta ^{q}}$ and
$ T ^{\sigma }:= \frac{1}{t -\zeta}$. Consequently, the poles of pair
$(T,\, T^{\sigma})$ are located at the truncated point of the projective
line. This notation for $T$ is consistent with its usage in the Carlitz
module. With these notations, $ \Psi $ and $\Psi ^{\sigma} $ are written
as
\begin{equation*}
\begin{cases}
\Psi _{x} = T^{\sigma -1} \tau ^{2} + (T^{\sigma}+T^{q-1+\sigma})
\tau + x
\\
\Psi _{y} = \zeta T^{\sigma -1} \tau ^{2} +(t T^{\sigma}+\zeta T^{q-1+
\sigma} )\tau +y,
\end{cases}
\end{equation*}
and
\begin{equation*}
\begin{cases}
\Psi _{x}^{\sigma }= T^{1-\sigma}\tau ^{2} + (T+T^{(q-1)\sigma +1})
\tau + x
\\
\Psi _{y}^{\sigma }= \zeta ^{q} T^{1-\sigma} \tau ^{2} +(t T+\zeta ^{q}T^{(q-1)
\sigma +1} )\tau + y.
\end{cases}
\end{equation*}
We then proved that the period lattice of $ \Psi $ is given by
\begin{equation*}
\frac{ \sqrt[q-1]{-[2]}}{[1]} \prod _{i = 1}^{\infty} \left (1-
\frac{[i]}{[i+2]} \right ) T^{\sigma -1} \cdot \mathcal{A},
\end{equation*}
where $[2k-1] = T^{q^{2k-1}}-T^{\sigma} $, $[2k ] = T^{q^{2k}}-T$ with
$k\geqslant 1$, and $\sqrt[q-1]{-[2]}$ denotes a $(q-1)$-th root of
$ -[2] $. 
Analogously, the period lattice of $ \Psi ^{\sigma }$ is given
by  
    \begin{equation*}
\frac{\sqrt[q-1]{-[2] T^{\sigma -1 }}}{[1]} \prod _{i = 1}^{\infty} \left (1-
\frac{[i]}{[i+2]} \right ) T^{\sigma }\cdot I_{\infty}^{-1},
\end{equation*}
    where $\sqrt[q-1]{-[2] T^{\sigma-1}}$ denotes a $(q-1)$-th root of
$ -[2] T^{\sigma -1 }$ and $I_{\infty}^{-1}  $ is a fractional ideal
of $ \mathcal{A}$ generated by $1,  \frac{y}{x} (= t) $.

The construction of rank two Drinfeld modules is involved with the Anderson's
$t$-motives \cite{AGW86}. For simplicity, we focus on the
$ \zeta ^{q} $-type and obtain the following complete family.
%
\begin{thm}
\label{Thm:cfPhiJ}
Let $ \nu $ be a $(q+1)$-th root of $ -T^{\sigma +q} $. There
exists a complete family $ \Phi ^{J} $ of rank two Drinfeld modules over
$ \mathcal{A} $-fields parameterized by $ J $, represented by the pair
$ (\Phi _{x}^{J}, \Phi _{y}^{J})$, where
%
\begin{align}
\nonumber
\Phi _{x}^{J} =& J^{q^{2}+q} \tau ^{4}+
\frac{J^{q^{2}+q}-J^{q+1}}{1-\zeta ^{1-q}} \tau ^{3}
\\
\label{Eq:PhiJx}
& +\Big(\frac{\zeta ^{q-1} J^{q+1}}{(1-\zeta ^{q-1})^{2}} +
\frac{\nu T^{q(q-1)(1-\sigma )-\sigma}J^{q}}{\zeta ^{q}-\zeta}-
\frac{T J}{(\zeta ^{q}-\zeta )\nu } +
\frac{T^{\sigma q-q}}{\zeta ^{q+1}} \Big) \tau ^{2}
\\
\nonumber
& +\Big(\frac{T^{\sigma q}+T^{\sigma}}{\zeta }+
\frac{(T^{\sigma +q}+x)J }{(1-\zeta ^{1-q})\nu}\Big) \tau +x,
\end{align}
and
%
\begin{align}
\nonumber
\Phi _{y}^{J} =& \zeta ^{q} J^{q^{2}+q} \tau ^{4}+ \zeta ^{q} \Big(
\frac{J^{q^{2}+q}-\zeta ^{1-q}J^{q+1}}{1-\zeta ^{1-q}} \Big)\tau ^{3}
\\
\label{Eq:PhiJy}
& + \zeta ^{q} \Big(\frac{ J^{q+1}}{(1-\zeta ^{q-1})^{2}}+
\frac{\nu T^{q(q-1)(1-\sigma )-\sigma}J^{q}}{\zeta ^{q}-\zeta} -
\frac{T J}{(\zeta ^{q}-\zeta )\nu } +
\frac{T^{\sigma q-q}}{\zeta ^{q+1}} \Big) \tau ^{2}
\\
\nonumber
& + \zeta ^{q} \Big(
\frac{\zeta ^{q} T^{\sigma q}+t T^{\sigma}}{\zeta ^{q+1}}+
\frac{(\zeta T^{\sigma +q}+y)J}{(\zeta ^{q}-\zeta )\nu }\Big) \tau +y.
\end{align}
The wedge product of $ \Phi ^{J} $, i.e.,
$ \Psi ^{J} = \wedge ^{2} \Phi ^{J} $ is represented by
\begin{equation*}
\begin{cases}
\Psi _{x}^{J} = -J^{q+1} \tau ^{2}-(\frac{x+T^{q+\sigma}}{\nu})J
\tau +x
\\
\Psi _{y}^{J} = -\zeta J^{q+1}\tau ^{2} -(
\frac{y+\zeta T^{q+\sigma}}{\nu})J\tau +y .
\end{cases}
\end{equation*}
\end{thm}

\subsection{Weil pairing of Drinfeld modules}
\label{sec1.2}

As a classical theory, the Weil pairing for elliptic curves $ E $ serves
as a perfect bilinear form from the $m$-torsion points ($m \geqslant 1
$) of $E$ to the $m$-th roots $ \mu _{m} $ of unity, formally written as
\begin{equation*}
E[m] \times E[m] \to \mu _{m} .
\end{equation*}
Given a Drinfeld module $ \phi $ over a fixed $A$-field and an ideal
$ I $ of $A $, the group $ \ker \phi _{I} $ equipped with rank $r$
$A$-module structure is the counterpart of torsion point group on elliptic
curves. Interpreting the image of the Weil pairing by
$\ker \psi _{I}$ for some rank one Drinfeld module $\psi $, there shall
be a perfect multilinear map
\begin{equation*}
\operatorname{Weil}_{I} :~ \prod _{i=1}^{r} \ker (\phi _{I} )\to
\ker (\psi _{I}).
\end{equation*}
We provide a brief interpretation for the relationship between Weil pairing
and Galois action. Let $K$ be the fraction field of $A$. Then the set of
torsions $\ker (\phi _{I})$ is an $A$-module, which is naturally equipped
with an action of the absolute Galois group
$ G_{K} = \mathrm{Gal}( K^{\mathrm{sep}}/ K)$. Since the Galois action
commutes with the $A$-action on $ \ker (\phi _{I}) $, we have the representation:
\begin{equation*}
\rho _{\phi}: G_{K} \to \mathrm{Aut}_{A}(\ker (\phi _{I})) \cong
\mathrm{GL}_{r} (A/I) .
\end{equation*}
Composing with the determinant map
$ \det : \mathrm{GL}_{r} (A/I)\to (A/I)^{*} = \mathrm{GL}_{1} (A/I)$, we
get a representation homomorphism
\begin{equation*}
G_{K} \to \mathrm{GL}_{1} (A/I).
\end{equation*}
Such a homomorphism shall be viewed as
\begin{equation*}
\rho _{\psi}: G_{K} \to \mathrm{Aut}_{A}(\ker \psi _{I})
\end{equation*}
for some rank one Drinfeld module $\psi $. In this way, the Weil pairing
shall be compatible with Galois actions.

Hamahata gave an explicit construction of rank two Weil pairing in
\cite{HY93} over the rational function field case. In general, the difficult
part to understand Weil pairing is that in the category of Drinfeld modules
it is not at all obvious how to define tensor products and how to take
subquotients. However, in the equivalent category of Anderson's pure
$t$-motives, the tensor product is closed under the operations of taking
subquotients and tensor products. In this way, van~der~Heiden
\cite{vdHGJ04} guaranteed the existence of Weil pairing. In particular,
he showed that for those Drinfeld module $ \phi $ of the form \text{\eqref{Eq:phioverpoly}}, the rank one Drinfeld module $ \psi $ can be represented
by
\begin{equation*}
\psi _{T} = (-1)^{r-1} g_{r} \tau + \iota (T) ,
\end{equation*}
which coincides with Hamahata's construction. Nevertheless, the Weil pairing
$\operatorname{Weil}_{I} $ can be formally characterized by the following
properties:
\begin{enumerate}
\item The map $\operatorname{Weil}_{I} $ is $A$-multilinear, i.e. it is
$A$-linear in each component.
\item It is alternating: if $\mu _{l} = \mu _{h}$ for $l \neq h$, then
$\operatorname{Weil}_{I}(\mu _{1},\ldots ,\mu _{r}) = 0$.
\item It is surjective and nondegenerate: if
\begin{equation*}
\operatorname{Weil}_{I}(\mu _{1},\ldots ,\mu _{r}) = 0 ~ {\mathrm{for~ all
}}~\mu _{1},\ldots , \mu _{i-1}, \mu _{i+1},\ldots , \mu _{r}\in \ker
\psi _{I},~ {\mathrm{then}} ~\mu _{i} = 0.
\end{equation*}
\item It is Galois invariant:
\begin{equation*}
\sigma \operatorname{Weil}_{I}(\mu _{1},\ldots ,\mu _{r}) =
\operatorname{Weil}_{I}(\sigma \mu _{1},\ldots ,\sigma \mu _{r}) ~{
\mathrm{for~ all}} ~\sigma \in \mathrm{Gal}(\bar K/K).
\end{equation*}
\item It satisfies the following compatibility condition for
$\mu _{1},\ldots , \mu _{r}\in \ker \psi _ {IJ}$:
\begin{equation*}
\psi _{J}\operatorname{Weil}_{IJ}(\mu _{1},\ldots ,\mu _{r}) =
\operatorname{Weil}_{I}(\phi _{J}(\mu _{1}),\ldots ,\phi _{J}(\mu _{r})).
\end{equation*}
\end{enumerate}
Using this alternative definition, J.~Katen \cite{KJ21} provided an explicit,
and elementary proof of the existence of Weil pairing for Drinfeld modules
over rational functional fields, which generalizes the rank two formula
in \cite{vdHGJ04}. More precisely, for the $I$-torsions
$ \mu _{1}, \cdots \mu _{n} $, the Weil pairing is written as
\begin{align*}
\operatorname{Weil}_{I} (\mu _{1},\ldots ,\mu _{r}) & = \left (\sum _{
\substack{{\mathbf i}=(i_{1},\ldots ,i_{r})\\0\leqslant i_{1},\ldots ,i_{r}\leqslant n-1}}a_{
\mathbf i } T_{1}^{i_{1}} T_{2}^{i_{2}} \cdots T_{n}^{i_{n}} \right ) *_{
\phi} (\mu _{1}, \ldots , \mu _{n} )
\\
& :=\sum _{
\substack{{\mathbf i}=(i_{1},\ldots ,i_{r})\\0\leqslant i_{1},\ldots ,i_{r}\leqslant n-1}}a_{
\mathbf i } \mathcal{M}(\phi _{T^{i_{1}}}(\mu _{1}),\ldots ,\phi _{T^{i_{r}}}(
\mu _{r})),
\end{align*}
where $\mathcal{M}$ denotes the Moore determinant, and
$ a_{\mathbf i} \in \mathbb{F}_{q} $ are coefficients depending on
$ I $. Indicated by this, we arrive at a nice formulation of Weil pairing
for rank $r$ Drinfeld modules over arbitrary base fields:
\begin{equation*}
\operatorname{Weil}_{I} (\mu _{1},\ldots ,\mu _{r}) =
\operatorname{WO}_{I}^{(r)} *_{\phi }(\mu _{1},\ldots , \mu _{r}),
\end{equation*}
where $ \operatorname{WO}_{I}^{(r)} $ denotes a non-trivial polynomial
with pairwise symmetric property. In particular, when
$ A = \mathcal{A}$, by constructing the dual basis of
$ \mathcal{A}/I $, we are able to derive an explicit expression of
$ \operatorname{WO}_{I}^{(2)} $, and refer to $ O'_{I} $ or
$ O''_{I} $ in \text{Notation~\ref{No:Wo}}.

This paper is organized as follows. First, in Section~\ref{Sec:rodm}, we
will deal with the explicit formula of rank one Drinfeld modules over
$\mathcal{A}$ and give a detailed discussion of the correspondent motive
structure. Using them, in Section~\ref{Sec:PL}, we compute the period lattices
of the standard models of rank one Drinfeld modules. In Section~\ref{Sec:RTDM}, we construct the modular curve (i.e., the $ J $-line) of
rank two Drinfeld modules. Finally, in Section~\ref{Sec:WP}, we discover
the explicit formula of Weil pairing.

\section{Rank one Drinfeld modules}
\label{Sec:rodm}

\subsection{Preliminaries}
\label{Sub:pre}

We recall briefly here some definitions and basic notions of Drinfeld modules.
Let $X$ be a nonsingular complete curve over a finite field
$\mathbb F_{q}$. Denote by $\infty $ a rational point of $X$, which will
be viewed as the infinity of $X$. Let $ A $ be the Dedekind domain corresponding
to $X\setminus \{\infty \}$. Denote by $d_{\infty} $ the degree of
$\infty $. Let $K $ be the fraction field of $ A $, which is also the function
field of $X$. Denote by $ K_ {\infty }$ the unique complete field with
respect to the $\infty $-valuation $ v_{\infty} $. The completion
$ \mathbb{C}_{\infty} $ of the algebraic closure of $ K_{\infty }$ plays
an analogous role, in the function field setting, of the complex number
field $\mathbb{C}$. Using the valuation $v_{\infty } $ at $\infty $, one
may define the degree of $a \in K $ by
\begin{equation*}
\deg (a) := -d_{\infty}\cdot v_{\infty}(a).
\end{equation*}
For $a \in A$, the degree of $a$ coincides with the cardinality of
$A/ (a) $.

The field $ L $ equipped with an $\mathbb{F}_{q}$-algebraic homomorphism
$ \iota : A \to L $ is called an $ A $-field. The $A$-characteristic of
$ L $ is defined to be
$ \operatorname{char}_{A}(L) = \ker (\iota ) $. If
$ \operatorname{char}_{A}(L) = 0 $, we say that $L$ has generic characteristic
or infinite characteristic in order to avoid confusion with the usual 0
characteristic.

Let $L\{\tau \}$ denote the twist polynomial ring, which is generated by
the $q$-th Frobenius endomorphism $\tau $ such that
$\tau a = a^{q} \tau $. There exists a natural action of
$ L\{\tau \} $ on $ \bar{L} $, given by
\begin{equation*}
\sum _{i=0}^{n} a_{i} \tau ^{i} : \ell \to \sum _{i=0}^{n} a_{i}
\ell ^{q^{i}}.
\end{equation*}

The Drinfeld module over the $ A $-field $ L $ is defined to be a non-trivial
algebraic homomorphism
\begin{equation*}
\phi : A \to L\{\tau \}, \quad a \mapsto \phi _{a},
\end{equation*}
such that the constant term of $ \phi _{a} $ equals $ \iota (a) $, where
the non-trivial restriction means $\phi _{a} \not \equiv \iota (a)$. If
the image of $\phi $ is contained in $ E\{\tau \} $ for some subfield
$E \subseteq L$, we say that $\phi $ is defined over $E$.

\begin{defn}%
\label{Defn:isogeny}
Let $ \lambda $ be a twist polynomial over $L$. Given Drinfeld modules
$ \phi $ and $ \phi ' $, we say that $ \lambda $ is an isogeny from
$ \phi $ to $ \phi ' $, if
\begin{equation*}
\lambda \phi _{a} = \phi '_{a} \lambda
\end{equation*}
holds for $ a \in A $. If $ \lambda $ is a nonzero constant polynomial,
then $ \lambda $ is called an isomorphism. In other words, Drinfeld module
$ \phi $ is isomorphic to its conjugate
$\lambda ^{-1} \phi \lambda $.
\end{defn}
As a not immediate fact, there exists a positive integer $r $ such that
%
\begin{align}
\label{Eq:degphi}
\deg _{\tau}(\phi _{a}) = r \deg (a),
\end{align}
for any $ a\in A $. The integer $ r $ in Equation \text{\eqref{Eq:degphi}} is
called the rank of $ \phi $. Via the Drinfeld module $\phi $, we have the
Drinfeld action of $A$ on $\bar L$, given by
%
\begin{align}
\label{Eq:actiona}
a: ~\ell \to \phi _{a}(\ell ),
\end{align}
for $a\in A$ and $\ell \in \bar L$. The set
$ \ker (\phi _{a}):= \{x\in \bar{L} : \phi _{a}(x) = 0 \}$ constitutes
the $a$-torsion points of $\phi $; it is in fact a submodule of
$\bar{L}$. If $a \in A$ is prime to the characteristic of $L $, then
\begin{equation*}
\ker (\phi _{a}) \cong (A / (a))^{\oplus r}
\end{equation*}
as $A$-modules. In analogy to the elliptic curve setting, the analytic
uniformization theorem establishes an equivalence between the categories
of lattices in $\mathbb{C}_{\infty}$, with morphisms coming from multiplications,
and Drinfeld modules over $ A $-field $ \mathbb{C}_{\infty}$, with isogeny
morphisms. In this way, each Drinfeld module over
$\mathbb{C}_{\infty}$ is associated to an exponential function
$\exp _{\Lambda}$ arising from some lattice $ \Lambda $ of
$\mathbb{C}_{\infty}$.

\begin{notation}%
\label{No:phiItphiI}
Let $I $ be an ideal of $A $. For a Drinfeld module $\phi $ over $L$, the
set
\begin{equation*}
\{\phi _{a}~| ~a\in I\}
\end{equation*}
form a left ideal of $L\{\tau \}$. As this ideal in principle, we denote
by $\phi _{I}$ the unique monic maximal right-divisor. Using this notation,
we have
%
\begin{equation}
\label{Eq:Annihilator}
\ker \phi _{I} = \{ x \in \ker \phi _{a} ~|~ \text{for each }a \in I
\}.
\end{equation}
The twist polynomial $ \phi _{I}$ is called the \textbf{annihilator} of
$ I $.
\end{notation}

\subsection{Domain}
\label{sec2.2}

We introduce the domain here which represents the truncated projective
line. Let $K = \mathbb{F}_{q}(t)$ be the rational function field over finite
field $\mathbb{F}_{q} $. As a basic fact, the places of $ K $ can be written
as
\begin{equation*}
\mathbb{P}_{K} = \{ P_{g} | g
\text{ is an irreducible monic polynomial over $ \mathbb{F}_{q}$}\}
\cup \{ \infty \}.
\end{equation*}

We consider the degree two place $ P_{\pi }$ which corresponds to the irreducible
monic polynomial $ \pi \in \mathbb{F}_{q}[t] $. In the rest of this paper,
we refer to $ P_{\pi }$ the infinity place of $K $, instead of the symbol
$\infty $.

Let $ \zeta \in \mathbb{F}_{q^{2}} $ be a root of $\pi $. Denote by
$\mathrm{Tr}= \mathrm{Tr}_{\mathbb{F}_{q}} :\mathbb{F}_{q^{2}} \to
\mathbb{F}_{q} $ the trace map. Then one may reformulate $ \pi $ as
\begin{equation*}
\pi = (t - \zeta )(t - \zeta ^{q} ) = t^{2} - \mathrm{Tr}(\zeta ) t +
\zeta ^{q+1},
\end{equation*}
where $\zeta ^{q+1}\in \mathbb{F}_{q}^{*}$ and
$\mathrm{Tr}(\zeta ) = \zeta + \zeta ^{q} \in \mathbb{F}_{q}$.

The coordinate domain of $ P_{\pi} $-truncated line is defined by
\begin{equation*}
\mathcal{A} := \{ u \in K ~|~ v_{Q}(u) \geqslant 0
\text{ for $Q \in \mathbb{P}_{K} \setminus \{ P_{\pi} \} $ } \}\, {
\cup \,\{0\}},
\end{equation*}
where $v_{Q}(u)$ denotes the valuation of $u$ at the place $Q$. Let
$\mathbb{F}_{\pi}$ be the residue field of $K$ at $ P_{\pi}$, which is
isomorphic to $ \mathbb{F}_{q}(\zeta ) \cong \mathbb{F}_{q^{2}}$. From
the definition, the fraction field of $ \mathcal{A} $ is identified with
the rational function field $ K = \mathbb{F}_{q}(t)$. For
$k\geqslant 1$, we view $k P_{\pi}$ as a divisor of $K$. From
\cite{S09} the Riemann-Roch space associated to $ k P_{\pi} $ is identical
to the $\mathbb F_{q}$-linear space
\begin{equation*}
\mathbb{L}(k P_{\pi}) := \{ u \in K ~ |~ v_{P_{\pi}}(u)\geqslant -k
\text{ and } v_{Q}(u) \geqslant 0
\text{ for $Q \in \mathbb{P}_{K} \setminus \{ P_{\pi} \} $ } \}\cup \{0
\} .
\end{equation*}
 Then $ \mathcal{A}$ can be expressed as
%
\begin{align}
\label{Eq:A=Lspan}
\mathcal{A}= \cup _{k=1}^{\infty} \mathbb{L}( k P_{\pi}).
\end{align}
Applying Riemann-Roch theorem, we find that for each
$ k \geqslant 1$, $ \mathbb{L}(k P_{\pi})$ is spanned by some monomials
in variables $x:= \frac{1}{\pi} $ and $ y: = \frac{t}{\pi } $. It follows
that $ \mathcal{A} $ is generated by $ x $ and $ y $ as an
$\mathbb{F}_{q}$-algebra. Indeed, by considering
$ x = \frac{1}{\pi } $ and $y= \frac{t}{\pi} $, we obtain the relation
of the form
%
\begin{align}
\label{Eq:rho}
\rho (x,y):=y^{2} - \mathrm{Tr}(\zeta ) x y + \zeta ^{q+1} x^{2} -x = 0
.
\end{align}
Consequently, there exists an algebraic isomorphism
\begin{equation*}
\mathcal{A}\cong \mathbb{F}_{q}[x,y]/ (\rho (x,y)) .
\end{equation*}
It yields that the elements of $ \mathcal{A}$ can be written as
$\frac{F(t)}{ \pi ^{d} }$ for some $ F(t) \in \mathbb{F}_{q}[t]$ with
$\deg (F(t)) \leqslant 2d $. For further study, we consider the two ideals
of $ \mathcal{A}$:
\begin{equation*}
I_{0}: = (x -\zeta ^{-q-1} , y ), \qquad I_{\infty }:= (x, y).
\end{equation*}
One may check that
\begin{equation*}
I_{\infty}^{2} = (x) \quad \text{ and }\quad I_{0} I_{\infty }= (y).
\end{equation*}
It is easy to see that the ideal class
\begin{equation*}
\mathrm{Cl}(\mathcal{A}) :=
\frac{\text{Ideal group of $\mathcal{A}$}}{\text{Principle ideal group of $\mathcal{A}$}}
\end{equation*}
is generated by $ I_{0} $ (or $I_{\infty}$), and therefore,
$ \mathrm{Cl}(\mathcal{A}) \cong \mathbb{Z}_{2} $.

\subsection{Hayes module}
\label{sec2.3}

According to the definition, a Drinfeld module of rank $r $ over an
$\mathcal{A}$-field is equivalent to the pair of twist polynomials
$(\phi _{x}, \phi _{y})$ satisfying that the conditions
\begin{align*}
\deg _{\tau} (\phi _{x}) & = \deg _{\tau}(\phi _{y}) = 2 r,
\\
\phi _{x} \phi _{y} & = \phi _{y} \phi _{x} ,
\end{align*}
and
\begin{equation*}
\phi _{y}\phi _{y} - \mathrm{Tr}(\zeta ) \phi _{x}\phi _{y} = \phi _{x}
-\zeta ^{q+1}\phi _{x} \phi _{x}.
\end{equation*}
For a Drinfeld module $ \phi $, denote by
$ \operatorname{LT}_{\phi} (a) $ the leading term of $ \phi _{a} $. It
is easy to see that
$ \operatorname{LT}_{\phi} (y)/ \operatorname{LT}_{\phi}(x) $ is a solution
of quadratic equation $ \pi (t)= 0 $, which leads to the following definition.
%
\begin{defn}
\label{defn2.3}
Given a Drinfeld module $ \phi $ over $ \mathcal{A}$, we say that
$ \phi $ is a $\zeta $-type (resp. $ \zeta ^{q} $-type) Drinfeld module
if $\operatorname{LT}_{\phi} (y)/ \operatorname{LT}_{\phi}(x) $ equals
$ \zeta $ (resp. $\zeta ^{q} $).
\end{defn}
Now we wish to compute the rank one and rank two Drinfeld modules over
$ \bar K $. We need the function field counterpart of the notion of number-theoretic
sign functions, see \cite[Definition 7.2.1]{Goss96}.
%
\begin{defn}%
\label{defn:sign}
Suppose that $ K_{\infty }$ is a $P_{\pi}$-adic  completion
of $ K $ containing $ \mathbb{F}_{\pi} $ as a subfield. A sign function
on $ K_{\infty}^{*} $ (resp. $K_{\infty}$) is a homomorphism
$ K_{\infty}^{*} \to \mathbb{F}_{\pi}^{*}$ which is the identity on
$ \mathbb{F}_{\pi}^{*} $.
\end{defn}

 Notice that $K_{\infty}$ is unique up to isomorphism. For
later use, we identify $\mathbb F_{\pi} $ with
$\mathbb F_{q}(\zeta )$, and fix the standard completion field
$K_{\infty}$ together with uniformizer $\frac{1}{T}=t-\zeta ^{q} $. So
we recognize that $t-\zeta ^{q}\in K_{\infty}$ has valuation $1$ and
$t-\zeta \in K_{\infty}$ has valuation $0$.

A sign function $\operatorname{Sgn}$ can be easily constructed in the following
fashion. Choose a local uniformizer $ \pi _{*} $ of place
$ P_{\pi} $. Then any element $f \in K_{\infty}^{*} $ admits a unique decomposition
\begin{equation*}
f = \operatorname{Sgn}(f) \pi _{*}^{v_{P_{\pi}}(f)} \langle f
\rangle _{1} ,
\end{equation*}
where $ \langle f \rangle _{1} $ is $1$-unit (i.e.,
$v_{P_{\pi}}(1-\langle f \rangle _{1}) \geqslant 1$), and
$\operatorname{Sgn}(f) \in \mathbb{F}_{\pi} ^{*} $. The function
$ \operatorname{Sgn}: K^{*} \to \mathbb{F}_{\pi}$ sending $ f $ to
$ \operatorname{Sgn}(f) $ is a sign function in the sense of \text{Definition~\ref{defn:sign}}. To be more precise, we choose $ \pi _{*} = \pi $. By identifying
$ \mathbb{F}_{\pi}$ with $ \mathbb{F}_{q}(\zeta ) $, the decompositions
\begin{equation*}
x = 1 \cdot \pi ^{-1} \cdot 1, \qquad y = \zeta ^{q} \cdot \pi ^{-1}
\cdot t/\zeta ^{q}
\end{equation*}
derive the sign function $\operatorname{Sgn}_{\zeta ^{q}}$ with values
\begin{equation*}
\begin{cases}
\operatorname{Sgn}_{\zeta ^{q}}(x) = 1
\\
\operatorname{Sgn}_{\zeta ^{q}}(y) = \zeta ^{q}.
\end{cases}
\end{equation*}
An easy calculation yields that there totally exist $(q^{2}-1)$ sign functions
on $K_{\infty}$, see \cite[Corollary 7.2.4]{Goss96}. Among these functions,
the one above with symbol $ \operatorname{Sgn}_{\zeta ^{q}} $ will be chosen
by default.
%
\begin{defn}
\label{defn2.5}
For a given sign function, e.g. $ \operatorname{Sgn}$, a Hayes module (with
respect to $ \operatorname{Sgn}$) over $ \mathcal{A}$-field is a rank one
Drinfeld module, say $\psi $, such that $ \operatorname{LT}_{\psi} $ is
a twist-sign function. In other worlds, there exists some
$\theta \in \mathrm{Gal}(\mathbb{F}_{\pi}/ \mathbb{F}_{q})$ such that
\begin{equation*}
\operatorname{LT}_{\psi}(a) = \theta \cdot \operatorname{Sgn}(a)
\end{equation*}
for each $a \in \mathcal{A}$.
\end{defn}
In our setting, as
$ \mathrm{Gal}(\mathbb{F}_{\pi} /\mathbb{F}_{q} ) = \{ \mathrm{Id},
\sigma \}$, the non-trivial twist of $ \operatorname{Sgn}_{
\zeta ^{q}} $ is just $ \operatorname{Sgn}_{\zeta}$ defined by identities
\begin{equation*}
\operatorname{Sgn}_{\zeta} (x) = 1 \text{ and } \operatorname{Sgn}_{
\zeta}(y) = \zeta .
\end{equation*}
Consequently, a rank one Drinfeld module $ \psi $ is a Hayes module (with
respect to $\operatorname{Sgn}_{\zeta ^{q}}$) if and only if
$\operatorname{LT}_{\psi }= \operatorname{Sgn}_{\zeta} $ or
$\operatorname{Sgn}_{\zeta ^{q}}$, corresponding to $\zeta $-type or
$\zeta ^{q}$-type.

Before we present the specific results of Drinfeld modules over
$\mathcal{A}$, we review the instructive facts concerning general case
of Drinfeld modules over general $ A $-fields. We refer the
reader to \cite[Proposition 7.2.20 and Chapter 7.4]{Goss96},
\cite[Theorem 7.5.21]{P23} or \cite{PB21} for more details.
%
\begin{thm}
\label{thm2.6}
The set of isomorphism classes of $A$-Drinfeld modules is the principle
homogeneous space of the class group of $A$, while the set of Hayes modules
is the principle homogeneous space of the narrow class group of $A$.
\end{thm}
\begin{thm}%
\label{thm:Hilbert}
Maintain the notations $\infty ,\, K $ in Preliminaries \ref{Sub:pre}.
The smallest definition field of a rank one Drinfeld module is the Hilbert
class field $ H$ of $ K $, which is the maximal unramified abelian extension
of $K$ in which $\infty $ splits completely. In addition, for a fixed sign
function $\operatorname{Sgn}$, the definition field of a Hayes module is
the narrow class field $H^{+}$ of $ K $ with respect to
$\operatorname{Sgn}$. The field $ H^{+} $ is totally and tamely ramified
over all infinities of $ H $, and the Galois group
$\mathrm{Gal}(H^{+} / H) $ is isomorphic to
$\mathbb F_{\infty}^{*}/ \mathbb F_{q}^{*} $, where
$ \mathbb F_{\infty}^{*} $ represents the multiplication group of the residue
field at $\infty $.
\end{thm}

For our setting, it is easy to see that the class field $ H$ of
$ \mathcal{A}$ is given by a constant field extension of $ K $, i.e.,
$ H = \mathbb{F}_{q^{2}}(t)= \mathbb{F}_{q}(\zeta ,t) $, whose Galois group
satisfies
\begin{equation*}
\mathrm{Gal}(H / K) \cong \mathrm{Cl}(\mathcal{A})\cong \mathbb{Z}_{2}.
\end{equation*}
Recall the definition of the narrow class group of $ \mathcal{A}$:
\begin{equation*}
\mathrm{Cl}^{+}(\mathcal{A}) :=
\frac{\text{Ideal group of $\mathcal{A}$}}{\text{Principle ideal group of positive elements in $\mathcal{A}$}},
\end{equation*}
where obviously the positive elements mean those with $
\operatorname{Sgn}_{\zeta ^{q}}$-value one. Class field theory indicates
the existence of narrow class field $ H^{+} $ associated with Galois group
$\mathrm{Cl}^{+}(\mathcal{A})$. For the sake of \text{Theorem~\ref{thm:Hilbert}}, when we try to derive a Hayes module over
$ \mathcal{A}$ in an explicit way, it is necessary to compute the narrow
class field $ H^{+} $ simultaneously.
%
\begin{prop}%
\label{Prop:HM}
Let $ \psi $ be a $\zeta $-type Hayes module and set
$ T = \frac{1}{t-\zeta ^{q}}$. Then $ \psi $ is given by
%
\begin{equation}
\label{Eq:psiExp}
\begin{cases}
\psi _{x} = \tau ^{2} - \frac{x+u^{q+1}}{u} \tau + x
\\
\psi _{y} = \zeta \tau ^{2}- \frac{y+u^{q+1}\zeta }{u} \tau +y ,
\end{cases}
%
\end{equation}
where $ u $ is a $(q+1)$-th root of $T^{\sigma +q}$; or equivalently,
\begin{equation*}
u^{q+1} = \frac{1}{(t - \zeta ) (t^{q} - \zeta ) } .
\end{equation*}
\end{prop}
\begin{proof}
Let $ \psi $ be such a Hayes module. In a precise form, we may assume that
\begin{equation*}
\begin{cases}
\psi _{x} = \tau ^{2} + a_{1} \tau + x
\\
\psi _{y} = \zeta \tau ^{2} + b_{1} \tau + y;
\end{cases}
\end{equation*}
with undetermined coefficients $ a_{1} $ and $ b_{1} $. Let
$ \psi _{I_{\infty}} = \tau - u $ be the annihilator of
$ I_{\infty }$. As both $ \psi _{x} $ and $ \psi _{y} $ are right-divisible
by $ \psi _{I_{\infty}}$, we derive
\begin{equation*}
a_{1} = \frac{-x-u^{q+1}}{u}
\end{equation*}
and
\begin{equation*}
b_{1} = \frac{-y-u^{q+1}\zeta}{u}.
\end{equation*}
Thus we obtain the expression of $\psi $ as in \text{\eqref{Eq:psiExp}}, and it
suffices to figure out the restrictions of $ u $. Indeed, the commutativity
$ \psi _{x} \psi _{y} = \psi _{y} \psi _{x} $ requires $ u $ to be a
$(q+1)$-th root of $ T ^{\sigma +q} $. In order to give a generalization
to the rank two case, we adopt another approach as follows.

Note that $ T = \frac{1}{t- \zeta ^{q}} = y - \zeta x $. It follows from
Equation \text{\eqref{Eq:psiExp}} that
%
\begin{align}
\label{Eq:psiyzex}
\psi _{y} + (\zeta ^{-q}- \zeta \psi _{x}) =\zeta ^{-q} -\frac{T}{u} (
\tau - u)= -\frac{T}{u} (\tau - u-\frac{u}{ \zeta ^{q} T} ).
\end{align}
Recalling $ I_{0} = (y, \zeta ^{-q-1}-x )$, we understand that the annihilator
$ \psi _{I_{0}} $ is the right-divisor of both $\psi _{y}$ and
$ \zeta ^{-q}- \zeta \psi _{x} $, see Equation \text{\eqref{Eq:Annihilator}}.
Now Equation \text{\eqref{Eq:psiyzex}} implies that
%
\begin{align}
\label{Eq:psiI0}
\psi _{I_{0}} = \tau - u-\frac{u}{\zeta ^{q} T} .
\end{align}
Again, from the fact that $ \psi _{I_{0}} $ is a right-divisor of
$ \psi _{y} = \zeta (\tau - \frac{y}{\zeta u })(\tau - u ) $, we have the
decomposition
%
\begin{align}
\label{Eq:psiypsiI0}
\psi _{y} = \zeta (\tau -\frac{\zeta ^{q} y}{\zeta t u}) \psi _{I_{0}}.
\end{align}
Then, substituting Equation \text{\eqref{Eq:psiI0}} in Equation \text{\eqref{Eq:psiypsiI0}}, we see that
%
\begin{align}
\label{Eq:psiyI0}
\zeta (\tau -\frac{\zeta ^{q} y}{\zeta t u}) (\tau -u)-\psi _{y} =
\zeta (\tau -\frac{\zeta ^{q} y}{\zeta t u}) \frac{u}{ \zeta ^{q} T}
\end{align}
is right divisible by $\tau -u$. So it yields that
\begin{equation*}
(\tau -\frac{\zeta ^{q} y}{\zeta t u}) \frac{u}{ \zeta ^{q} T} = (
\frac{ u }{\zeta ^{q} T})^{q} ( \tau - u ).
\end{equation*}
Therefore, comparing constant terms of both sides, we have
\begin{equation*}
u^{q+1}=T^{\sigma +q}.
\end{equation*}
One may directly check that
\begin{equation*}
\psi _{y - \mathrm{Tr}(\zeta ) x} \psi _{y} + \psi _ { \zeta ^{q+1} x-1
} \psi _{x} = 0
\end{equation*}
and
\begin{equation*}
\psi _{x} \psi _{y} = \psi _{y} \psi _{x}.
\end{equation*}
This implies that $ \psi $ is indeed a Drinfeld module of rank one.
\end{proof}

As a consequence, the definition field of $ \psi $ is
$ H^{+} = K(\zeta ,u )$ with
\begin{equation*}
u^{q+1} = \frac{1}{(t -\zeta )(t^{q} - \zeta ) } .
\end{equation*}
A direct argument also implies that $ H^{+} $ is the narrow class field
of $ K $, with extension degree
\begin{equation*}
[H^{+}: K] = 2 \cdot (q+1)
\end{equation*}
and
\begin{equation*}
[H^{+}: H] = q+1=[\mathbb F_{q^{2}}^{*}: \mathbb F_{q}^{*}].
\end{equation*}
 All these facts verify \text{Theorem~\ref{thm:Hilbert}}. Essentially,
$ H^{+} $ is the splitting field of the polynomial
\begin{equation*}
F^{+} = X^{2(q+1)} -
\frac{2t^{q+1} - \mathrm{Tr}\zeta \cdot (t^{q}+t) + \mathrm{Tr}(\zeta )^{2} - 2 \zeta ^{q+1} }{\pi ^{q+1} }
X^{q+1} + \frac{1}{\pi ^{q+1} } \in K[X].
\end{equation*}
If $ u_{0} $ is one of roots of $ F^{+} $, then all the roots of
$ F^{+} $ can be expressed as
$ \left \{ \mu _{1} u_{0} , \, \frac{ \mu _{2} x}{u_{0}} \right \}_{
\mu _{1},\,\mu _{2}} $, where both $\mu _{1}$ and $ \mu _{2}$ represent
$(q+1)$-th roots of unity. Therefore, the Galois group of
$ H^{+} / K $ is isomorphic to the narrow class group
$ \mathrm{Cl}^{+}(K)$, represented as
\begin{equation*}
\begin{cases}
\zeta \mapsto \zeta
\\
u \mapsto u \cdot \mu _{1};
\end{cases}
\quad
\begin{cases}
\zeta \mapsto \zeta ^{q}
\\
u \mapsto \frac{x}{u} \cdot \mu _{2}.
\end{cases}
\end{equation*}
%
\begin{defn}%
\label{Defn:psiu}
For a root $u$ of $ F ^{+} $, define the Drinfeld module
$ \psi ^{u} $ by the pair:
\begin{equation*}
\begin{cases}
\psi _{x}^{u} = (\tau - \frac{x}{u})(\tau - u)
\\
\psi _{y}^{u} = (\zeta _{*} \tau - \frac{y}{ u}) (\tau - u ),
\end{cases}
\end{equation*}
where $ \zeta _{*} $ depending on $ u $ is a root of $ \pi $ subjecting
to
%
\begin{equation}
\label{Eq:uzeta}
u^{q+1} = \frac{1}{(t-\zeta _{*}) (t^{q} - \zeta _{*})} .
\end{equation}
In other words, such $\psi ^{u} $ is a Hayes module of $ \zeta _{*}$-type.
\end{defn}
As a consequence, we obtain the following refinement of \text{Proposition~\ref{Prop:HM}}.
%
\begin{prop}%
\label{Prop:hayes}
For the fixed sign function $ \operatorname{Sgn}_{\zeta ^{q}}$,
the set of Hayes modules over $\mathcal{A}$-field
$ \mathbb{C}_{\infty} $ is given as
\begin{equation*}
\{ \psi ^{u} | \text{ $u$ is a zero of $ F^{+} $} \}.
\end{equation*}
The twist polynomial $ \psi _{I_{\infty}} $ gives an isogeny from
$ \psi ^ {u} $ to $ \psi ^{u'} $, with $ u ' = \frac{x}{u} $, and exchanges
the types.
\end{prop}
As the isomorphism class of $ \phi ^{u} $ is given by the conjugate form
$ \ell ^{-1} \phi ^{u} \ell $ for some $ \ell \in \bar{L}$, we can define
the complete family of rank one Drinfeld modules as follows.
%
\begin{defn}%
\label{Defn:psize*}
Let $ \zeta _{*} $ be a root of $ \pi (t) $. For coefficient
$a \in L $, define the Drinfeld module
$\psi ^{(\zeta _{*},a)}$ as follows:
%
\begin{align}
\label{Eq:defpsi}
\begin{cases}
\psi _{x}^{(\zeta _{*}, a)} :=
\frac{t - \zeta _{*}^{q} }{t - \zeta _{*} } \cdot \frac{1}{a^{q+1}}
\tau ^{2} +
\frac{t^{q}+t - \mathrm{Tr}\zeta _{*} }{ (t - \zeta _{*})(t^{q}- \zeta _{*}) }
\frac{1}{a} \tau +x
\\
\psi _{y}^{(\zeta _{*}, a) } :=
\frac{t - \zeta _{*}^{q} }{t - \zeta _{*} } \cdot
\frac{\zeta _{*}}{a^{q+1}} \tau ^{2} +
\frac{t^{q+1} - \zeta _{*}^{q+1} }{ (t - \zeta _{*})(t^{q}- \zeta _{*}) }
\frac{1}{a}\tau + y.
\end{cases}
\end{align}
\end{defn}
Among the family, there is a $\zeta $-type Drinfeld module
$ \Psi := \psi ^{(\zeta , 1)}$ with coefficients in $H$ written as
\begin{equation*}
\begin{cases}
\Psi _{x} =T^{\sigma -1} \tau ^{2} + (T^{\sigma}+T^{q-1+\sigma})\tau +
x
\\
\Psi _{y} = \zeta T^{\sigma -1} \tau ^{2} +(t T^{\sigma}+ \zeta T^{q-1+
\sigma} )\tau +y.
\end{cases}
\end{equation*}
While the $\zeta ^{q} $-type rank one Drinfeld modules are isomorphic to
the twist of $ \Psi $, denoted by $ \Psi ^{\sigma} $. We call $
\Psi $ and $\Psi ^{\sigma}$ the standard models of rank one Drinfeld modules
over $ \mathcal{A}$-fields.

\subsection{Motive structure of Drinfeld modules}
\label{sec2.4}

In light of Anderson's $t$-motives, a Drinfeld module is equivalently represented
by an $\mathcal{A}\otimes _{\mathbb{F}_{q}} K\{ \tau \}$-module. For
simplicity, we adopt the definition below.
%
\begin{defn}
\label{defn2.12}
Let $\bar K$ be an algebraic closure of $K$. For a Drinfeld module
$ \phi $ over $\bar K$, the twist polynomial ring
$\bar{M}(\phi ) = \bar{K}\{\tau \} $ over $ \bar K $ equipping the left
$ \mathcal{A}\otimes _{\mathbb{F}_{q}} \bar{K}\{ \tau \} $-module structure:
\begin{equation*}
a \otimes g(\tau ) \cdot \tau ^{k} = g(\tau ) \cdot \tau ^{k} \cdot
\phi _{a} \quad
\text{ for
$ a\otimes g(\tau ) \in \mathcal{A}\otimes _{\mathbb F_{q}} \bar{K}
\{ \tau \}, ~ \tau ^{k} \in \bar{M}(\phi ) $}
\end{equation*}
is called the motive structure of $ \phi $. Notice that
$ \bar{M}(\phi ) $ is then a projective
$ \mathcal{A}\otimes _{\mathbb{F}_{q}}\bar{K} $-module of rank $r$ with
generators $1_{\phi}$, $\tau $, $\tau ^{2}$, $\cdots $.
\end{defn}

For rank one Drinfeld modules, we are able to calculate the associated
motive structures explicitly.
%
\begin{prop}%
\label{Prop:psimo}
Let $ \tilde{\psi} $ be a rank one Drinfeld module over $ \bar {K} $ whose
motive structure is given by the relation
\begin{equation*}
(-c + y \otimes a - x \otimes b) \cdot 1_{\tilde \psi} = \tau ,
\end{equation*}
where $ a ,~ b, ~c \in \bar K $. Then $\zeta _{*} = b/ a $ becomes a root
of $ \pi (t) = 0 $, and $ a,\, b,\, c $ verify the relations
\begin{equation*}
\frac{a}{c} = t - \zeta _{*}^{q}, \qquad \frac{b}{c} = \zeta _{*}(t -
\zeta _{*}^{q}).
\end{equation*}
Moreover, $\tilde \psi $ is identical to the form
$ \psi ^{(b/a, a)}$ in \text{Definition~\ref{Defn:psize*}}.
\end{prop}
\begin{proof}
Let $ \psi ^{u} $ be the Hayes module as in \text{Definition~\ref{Defn:psiu}}. From \text{Proposition~\ref{Prop:hayes}}, we can assume that
there exists some $ \ell \in \bar {K} $, such that
\begin{equation*}
\tilde \psi _{f} = \ell ^{-1} \psi _{f}^{u} \ell
\end{equation*}
for any $ f\in \mathcal{A}$. Precisely, we have
\begin{align*}
\begin{cases}
\tilde{\psi}_{x} = \ell ^{q^{2}-1} \tau ^{2} - \ell ^{q-1}
\frac{x+u^{q+1}}{u} \tau +x
\\
\tilde{\psi}_{y} = \ell ^{q^{2}-1}\zeta _{*} \tau ^{2}-\ell ^{q-1}
\frac{y+\zeta _{*} u^{q+1}}{u}\tau + y ,
\end{cases}
\end{align*}
where $ \zeta _{*} $ is a root of $\pi $ satisfying the equality \text{\eqref{Eq:uzeta}}. From the motive structure of $\tilde{\psi}$, the relation
of $ 1_{\tilde\psi } ,~\tau , ~\tau ^{2} $ satisfies
\begin{equation*}
\begin{cases}
(x \otimes 1) \cdot 1_{\tilde{\psi}}= \ell ^{q^{2}-1} \cdot \tau ^{2} -
\ell ^{q-1} \frac{x+u^{q+1}}{u} \cdot \tau + x \cdot 1_{\tilde{\psi}}
\\
(y \otimes 1) \cdot 1_{\tilde{\psi}}= \ell ^{q^{2}-1}\zeta _{*}\cdot
\tau ^{2}-\ell ^{q-1}\frac{y+\zeta _{*} u^{q+1}}{u}\cdot \tau + y
\cdot 1_{\tilde{\psi}} .
\end{cases}
\end{equation*}
Note that we write $ x,y \in A\otimes \bar K $ by abuse of notation, since
$\bar K $ is contained in $ A\otimes \bar K $. Combining these equations,
we have the equality
\begin{equation*}
\frac{u}{\ell ^{q-1}} \cdot 1_{\tilde\psi} - (y \otimes
\frac{u(t-\zeta _{*}^{q})}{\ell ^{q-1}}) \cdot 1_{\tilde\psi} +(x
\otimes \frac{\zeta _{*} u(t-\zeta _{*}^{q})}{\ell ^{q-1}} ) \cdot 1_{
\tilde\psi} = \tau .
\end{equation*}
The proposition follows by the assignments
\begin{equation*}
a:= -\frac{u(t-\zeta _{*}^{q})}{\ell ^{q-1}}, \quad b:=-
\frac{\zeta _{*} u(t-\zeta _{*}^{q})}{\ell ^{q-1}} , \quad c:=-
\frac{u}{\ell ^{q-1}}.\qedhere
\end{equation*}
\end{proof}

\section{Period lattice of standard Drinfeld modules}
\label{Sec:PL}

In complex analysis, one can associate to an elliptic curve $E$ over
$\mathbb{C}$ a rank two lattice $\Lambda \subseteq \mathbb{C}$, from which
we can define the Weierstrass $ \mathcal{P}$-function. This gives rise
an exact sequence
\begin{equation*}
\begin{tikzcd}
0 \ar[r] & \Lambda \ar[r] & \mathbb{C} \ar[r, "\exp _{E}"] &E(
\mathbb{C}) \ar[r]& 0,
\end{tikzcd}
\end{equation*}
where $ \exp _{E} = (\mathcal{P}, \mathcal{P}', 1 )$. The addition rule
on complex-point set $ E(\mathbb{C})$ is then derived by the addition on
$ \mathbb{C} $. In function field arithmetic, a rank two Drinfeld module
plays an analogous role of elliptic curve. Let us fix the
$\mathcal{A}$-field $ \mathbb{C}_{\infty}$, we understand from the last
section that each rank one Drinfeld module over
$ \mathbb{C}_{\infty} $ has the standard model, either $\Psi $ or
$ \Psi ^{\sigma} $. Let us focus on the $\zeta $-type first. The analytic
uniformization theorem yields that $ \Psi $ is described by an exponential
function, denoted by $ \exp _{\Lambda} $, whose kernel $ \Lambda $ reserves
as an $\mathcal{A}$-lattice in $ \mathbb{C}_{\infty} $. As shown in the
following diagram of exact sequences
\begin{equation*}
\begin{tikzcd}
0 \ar[r] & \Lambda \ar[r] \arrow[d, "f"] & \mathbb{C}_{\infty}
\ar[d, "f" ] \ar[r, "\exp _{\Lambda}"] &\mathbb{C}_{\infty}
\ar[d, "\Psi _{f}"] \ar[r]& 0
\\
0 \ar[r] & \Lambda \arrow[r] & \mathbb{C}_{\infty}
\ar[r, "\exp _{\Lambda}"] & \mathbb{C}_{\infty} \ar[r] & 0
\end{tikzcd}
\end{equation*}
for each $f\in \mathcal{A}$, the $f$-multiplication on $ \Lambda $ recovers
the Drinfeld action $ \Psi _{f} $ on $ \mathbb{C}_{\infty} $. Since
$ \operatorname{Rank}(\Lambda )= \operatorname{Rank}(\Psi ) = 1 $, we may
assume that
\begin{equation*}
\Lambda = \xi _{\mathcal{A}} \cdot \mathcal{A}
\end{equation*}
for some quantity $\xi _{\mathcal{A}}\in \mathbb{C}_{\infty}$. The quantity
$\xi _{\mathcal{A}}$ shall be understood to be the analogy of complex number
$ 2 \pi i $, and called Carlitz period (see
\cite[Chapter 3.2]{Goss96}, \cite[Chapter 2.5]{Tha04}). For the rational
function field $ \mathbb{F}_{q}(t)$, Carlitz formulated the period as an
infinite product. Wade \cite{Wade41} showed that the correspondent Carlitz
period is transcendental over $ \mathbb{F}_{q}(t) $. We would like to determine
the period lattice of $ \Psi $ explicitly in the same fashion.

\subsection{Binormial coefficients}
\label{sec3.1}

Let us introduce the binormial coefficients in function field arithmetic.
%
\begin{notation}%
\label{No:Dj}
For $ j \geqslant 0$, the sequence $ \{ D_{j} \} $ of
$ K(\zeta ) = \mathbb{F}_{q}(t,\zeta ) $ is recursively defined as follows:
$D_{0} = 1 $, and
\begin{equation*}
D_{j} = D_{j-1}^{q}
\frac{t-t^{q^{j}}}{(t-\zeta ^{q^{j-1}})^{q^{j}}(t-\zeta ^{q})}.
\end{equation*}
\end{notation}

The first few terms are given by
%
\begin{align}\label{eq18}
D_{1} = \frac{t- t^{q}}{(t - \zeta )^{q} (t- \zeta ^{q})};
D_{2}=
\frac{(t-t^{q})^{q} (t - t^{q^{2}})}{(t - \zeta )^{q^{2}} (t - \zeta ^{q})^{q^{2}+q+1}};
D_{3} =
\frac{(t-t^{q})^{q^{2}} (t-t^{q^{2}})^{q} (t -t^{q^{3}})}{ (t -\zeta )^{2q^{3}} (t - \zeta ^{q})^{q^{3}+q^{2}+q+1}}.
\end{align}
Let $ M_{j}, N_{j} $ be the quantities
\begin{equation*}
\begin{cases}
M_{j} = \lceil \frac{j}{2} \rceil q^{j}
\\
N_{j} = \lfloor \frac{j}{2} \rfloor q^{j} + \frac{q^{j} -1}{q-1}.
\end{cases}
\end{equation*}
Here and in the sequel, we denote by $\lceil - \rceil $ the ceiling and
by $\lfloor - \rfloor $ the floor. Then $D_{j}$ in \text{Notation~\ref{No:Dj}} can be rewritten as
\begin{equation*}
D_{j} =
\frac{(t-t^{q})^{q^{j-1}} (t-t^{q^{2}})^{q^{j-2}} \cdots (t -t^{q^{j}})}{ (t -\zeta )^{M_{j}} (t - \zeta ^{q})^{N_{j}}}.
\end{equation*}
%
\begin{notation}%
\label{No:Lj}
Let $L_{0} = 1 $. For $ j \geqslant 1 $, we define
$ L_{j} \in K(\zeta )$ recursively by the formula:
\begin{equation*}
L_{j} = L_{j-1}
\frac{t-t^{q^{j}}}{ (t - \zeta ^{q^{j+1}})(t-\zeta )^{{q^{j}}-q^{j-1}} (t - \zeta ^{q})^{q^{j-1}}}.
\end{equation*}
\end{notation}

Using the quantities $ M_{j}',~ N_{j}' $:
\begin{equation*}
\begin{cases}
M'_{j} = q^{j}+ \lceil \frac{j}{2} \rceil -1
\\
N'_{j} = \lfloor \frac{j}{2} \rfloor + \frac{q^{j}-1}{q-1} ,
\end{cases}
\end{equation*}
each term $L_{j}$ is given by
%
\begin{equation}
\label{Eq:L-exp}
L_{j} =
\frac{(t-t^{q}) (t-t^{q^{2}}) \cdots (t -t^{q^{j}})}{(t - \zeta )^{M_{j}'} (t - \zeta ^{q})^{N_{j}'}}.
\end{equation}

\begin{notation}
For integer $ k \geqslant 0 $, we adopt the symbols
\begin{equation*}
\begin{cases}
[2k] := T^{q^{2k}} - T
\\
[2k - 1] := T^{q^{2k-1}} - T^{\sigma }.
\end{cases}
\end{equation*}
\end{notation}
It is trivial to check that
%
\begin{align}
\label{Eq:ttq2k}
t - t^{q^{2k}} = \frac{T^{q^{2k}} -T }{T^{q^{2k}+1 }}=
\frac{[2k]} {T^{q^{2k}+1 }}
\end{align}
and
%
\begin{align}
\label{Eq:ttq2k-1}
t - t^{q^{2k-1}} =
\frac{T^{q^{2k-1}}- T^{\sigma}}{T^{q^{2k-1 + \sigma}}}=
\frac{[2k-1]}{T^{q^{2k-1 + \sigma}}}.
\end{align}
With these notations, we can obtain the following lemma.
%
\begin{lem}%
\label{Lem:leLj}
Each term $ L_{j} $ in \text{Notation~\ref{No:Lj}} can be reformulated as
\begin{equation*}
L_{j} = [1][2][3]\cdots [j] T^{(\sigma -1)(q^{j}-1)}.
\end{equation*}
\end{lem}
\begin{proof}
Substituting Equations \text{\eqref{Eq:ttq2k}} and \text{\eqref{Eq:ttq2k-1}} into the
expressions of $L_{2k}$ and $ L_{2k+1}$ in Equation \text{\eqref{Eq:L-exp}} respectively,
we have
\begin{equation*}
\begin{split} L_{2k}&= \frac{[1]}{T^{q+\sigma}}
\frac{[2]}{T^{q^{2}+1}} \cdots \frac{[2k-1]}{T^{q^{2k-1}+\sigma}}
\frac{[2k]}{T^{q^{2k}+1}}T^{ M_{2k}' \sigma + N_{2k}'}
\\
&= [1][2][3]\cdots [2k] T^{ (M_{2k}'- k)\sigma +N_{2k}'- k -
\frac {q^{2k+1}-q}{q-1}}
\end{split}
\end{equation*}
and
\begin{equation*}
\begin{split} L_{2k+1}&= \frac{[1]}{T^{q+\sigma}}
\frac{[2]}{T^{q^{2}+1}} \cdots \frac{[2k]}{T^{q^{2k}+1}}
\frac{[2k+1]}{T^{q^{2k+1}+\sigma}}T^{ M_{2k+1}' \sigma + N_{2k+1}'}
\\
&= [1][2][3]\cdots [2k+1] T^{ (M_{2k+1}'- k-1)\sigma +N_{2k+1}'- k -
\frac {q^{2k+2}-q}{q-1}}.
\end{split}
\end{equation*}
Thus, combining the last two formulas, we have
\begin{align*}
L_{j} =& [1][2][3]\cdots [j] T^{ (M_{j}'- \lceil \frac{j}{2} \rceil )
\sigma +N_{j}'- \lfloor \frac{j}{2} \rfloor -\frac {q^{j+1}-q}{q-1}}
\\
=& [1][2][3]\cdots [j]T^{(\sigma -1)(q^{j}-1)}.\qedhere
\end{align*}
\end{proof}

\subsection{Exponential function and logarithm function}
\label{sec3.2}

It is convenient to identify $\mathbb{C}_{\infty}$ with the completion
$ \mathbb{C}_{\zeta ^{q} } $ of the algebraic closure of
$ \mathbb{F}_{q}(\zeta )((\frac{1}{T})) $ with respect to the valuation
$ v_{\frac{1}{T}} $ located at the zero locus of
$ \frac{1}{T} = t - \zeta ^{q} $. Let
$ f: \mathbb{C}_{\zeta ^{q}} \to \mathbb{C}_{\zeta ^{q}} $ be an
$\mathbb{F}_{q}$-linear function over $ \mathbb{C}_{\zeta ^{q}}$. It follows from \cite[Corollary 13.2.5]{V-S06} that such function is given
by the series
\begin{equation*}
f(X) = a_{0} X + a_{1} X^{q} + a_{2} X^{q^{2}} + \cdots .
\end{equation*}
The first coefficient $a_{0} $ is the derivative of $f$. Let
$ \Lambda $ be the lattice of $ \Psi $ contained in
$ \mathbb{C}_{\zeta ^{q}} $. The exponential function and the logarithm
function associated to $ \Lambda $ can be formally defined as follows.
%
\begin{defn}%
\label{Defn:exponential}
The exponential function
$\exp _{\Lambda} (X): \mathbb{C}_{\zeta ^{q}} \to \mathbb{C}_{\zeta ^{q}}
$ is an $\mathbb F_{q}$-linear function subject to
\begin{equation*}
\begin{cases}
\Psi _{x}\cdot \exp _{\Lambda}(X) = \exp _{\Lambda} \cdot (x X)
\\
\Psi _{y}\cdot \exp _{\Lambda}(X)= \exp _{\Lambda} \cdot (y X),
\end{cases}
\end{equation*}
whose derivative is identically $1$.
\end{defn}

\begin{defn}
\label{defn3.6}
The logarithm function $\log _{\Lambda}(X)$ is the local inverse of the
exponential function, which satisfies
%
\begin{align}
\label{Eq:delog}
\begin{cases}
\log _{\Lambda} (\Psi _{x} (X) ) = x \log _{\Lambda}(X)
\\
\log _{\Lambda} ( \Psi _{y} (X) ) = y \log _{\Lambda}(X).
\end{cases}
\end{align}
\end{defn}

\begin{prop}%
\label{Prop:explog}
Let $ \tau $ be the $q$-th Frobenius map. Using the binormial coefficients
$ D_{j}, ~L_{j} $ in \text{Notations~\ref{No:Dj} and \ref{No:Lj}}, the exponential
and logarithm functions are written as follows:
%
\begin{equation}
\label{Eq:expLambda}
\exp _{\Lambda} (X) = \sum _{j=0}^{\infty }\frac{X^{q^{j}}}{D_{j}} =
\sum _{j=0}^{\infty }\frac{\tau ^{j}}{D_{j}}(X)
\end{equation}
and
%
\begin{align}
\label{Eq:logLambda}
\log _{\Lambda}(X) = \sum _{j=0}^{\infty}
\frac{(-1)^{j} X^{q^{j}}}{L_{j}} = \sum _{j=0}^{\infty }
\frac{(-1)^{j} \tau ^{j} }{L_{j}} (X).
\end{align}
\end{prop}
\begin{proof}
Let us assume that
\begin{equation*}
\exp _{\Lambda} (X) = \sum _{j=0}^{\infty }
\frac{X^{q^{j}}}{D_{j}^{*}} = \sum _{j=0}^{\infty }
\frac{\tau ^{j}}{D_{j}^{*}}(X)
\end{equation*}
for some $ D^{*}_{j} \in \mathbb{C}_{\zeta ^{q}}$ and $D_{0}^{*}=1$. We
wish to show that $D_{j}^{*} = D_{j} $ for each $ j $. From the expression
of $\Psi $ and the relation
\begin{equation*}
\Psi _{y}- \zeta \Psi _{x} = \tau +T ,
\end{equation*}
we have
%
\begin{align}
(\Psi _{y}- \zeta \Psi _{x})\cdot \exp _{\Lambda}(X) &= \exp _{
\Lambda}(X)^{q} + T \exp _{\Lambda}(X)
\nonumber
\\
& = \sum _{j=1}^{\infty }\frac{X^{q^{j }}}{(D_{j-1}^{*})^{q}} + T
\sum _{j=0}^{\infty }\frac{X^{q^{j}}}{D_{j}^{*}}
\label{Eq:Psiexp1}
.
\end{align}
On the other hand, we obtain from \text{Definition~\ref{Defn:exponential}} that
%
\begin{align}
(\Psi _{y}- \zeta \Psi _{x})\cdot \exp _{\Lambda}(X) &= \exp _{
\Lambda} \cdot (y X) - \zeta \exp _{\Lambda} \cdot (x X)
\nonumber
\\
& = \sum _{j=0}^{\infty }(y^{q^{j}} -\zeta x^{q^{j}})
\frac{X^{q^{j}}}{D_{j}^{*}}
\label{Eq:Psiexp2}
.
\end{align}
Combining Equations \text{\eqref{Eq:Psiexp1}} and \text{\eqref{Eq:Psiexp2}} that
\begin{equation*}
\sum _{j=1}^{\infty }\frac{X^{q^{j }}}{(D_{j-1}^{*})^{q}} + T \sum _{j=0}^{
\infty }\frac{X^{q^{j}}}{D_{j}^{*}} = \sum _{j=0}^{\infty }(y^{q^{j}} -
\zeta x^{q^{j}}) \frac{X^{q^{j}}}{D_{j}^{*}} .
\end{equation*}
Comparing the coefficients of both sides, we find that $ D_{j}^{*} $ satisfies
the recursive relation
\begin{equation*}
\frac{1}{(D_{j-1}^{*})^{q}} =
\frac{y^{q^{j}} - x^{q^{j}} \zeta - T }{D_{j}^{*}}.
\end{equation*}
That is
\begin{equation*}
D_{j}^{*} = (D_{j-1}^{*})^{q} (y^{q^{j}} - x^{q^{j}} \zeta - T)= (D_{j-1}^{*})^{q}
\frac{t-t^{q^{j}}}{(t-\zeta ^{q^{j-1}})^{q^{j}}(t-\zeta ^{q})} ,
\end{equation*}
which coincides with recursive relation of $ D_{j} $. So
$ D_{j} = D^{*}_{j} $.

In order to compute $ \log _{\Lambda }$, we assume that
\begin{align*}
\log _{\Lambda}(X) = \sum _{j=0}^{\infty}
\frac{(-1)^{j} X^{q^{j}}}{L_{j}^{*}} = \sum _{j=0}^{\infty }
\frac{(-1)^{j} \tau ^{j} }{L_{j}^{*}} (X),
\end{align*}
for some $ L^{*}_{j} $ and $L_{0}^{*}=1$. It suffices to show that
$L_{j}^{*} = L_{j} $ for each $ j $.

Let us check $ L_{1}^{*} $ first. Using the expressions of
$\Psi _{x}$ and $ \log _{\Lambda}$, the first formula in Equations \text{\eqref{Eq:delog}} can be rewritten as
\begin{equation*}
\left ( \sum _{j=0}^{\infty} \frac{(-1)^{j} \tau ^{j}}{L_{j}^{*}}
\right ) \cdot \left ( x + \frac{x+T^{\sigma +q}}{T} \tau + T^{
\sigma -1} \tau ^{2} \right ) = x \sum _{j=0}^{\infty}
\frac{(-1)^{j} \tau ^{j}}{L_{j}^{*}} ,
\end{equation*}
i.e.,
%
\begin{align}
\label{Eq:logaddL}
\sum _{j=0}^{\infty }\frac{ (-1)^{j}(x^{q^{j}} - x) }{ L_{j}^{*}}
\tau ^{j} + \sum _{j=0}^{\infty }
\frac{ (-1)^{j} (x^{q^{j}} + T^{(\sigma +q)q^{j}})}{ T^{q^{j}} L_{j}^{*} }
\tau ^{j+1} +\sum _{j=0}^{\infty }
\frac{ (-1)^{j} T^{(\sigma -1)q^{j}} }{ L_{j}^{*} } \tau ^{j+2} = 0.
\end{align}
Since $ L_{0}^{*} = 1 $, the coefficient of $\tau $ in Equation \text{\eqref{Eq:logaddL}} equals
\begin{equation*}
- \frac{ x^{q} - x }{ L_{1}^{*} } \tau +
\frac{ (x + T^{\sigma +q} ) }{ T L_{0}^{*} } \tau = 0.
\end{equation*}
Thus, we have
%
\begin{align}
\label{Eq:L1L0}
 L_{1}^{*} =\frac{T (x^{q}-x)}{x+T^{\sigma +q}} =
\frac{t-t^{q}}{(t^{q}-\zeta ^{q})(t - \zeta ^{q})} .
\end{align}
Now we come to compute $ L_{j}^{*} $ for $j>1$. Equation \text{\eqref{Eq:logaddL}} can be rewritten as
\begin{equation*}
\sum _{j=2}^{\infty }\frac{ (-1)^{j}(x^{q^{j}} - x) }{ L_{j}^{*} }
\tau ^{j} - \sum _{j=2}^{\infty }
\frac{ (-1)^{j} (x^{q^{j-1}} + T^{(\sigma +q)q^{j-1}}) }{ T^{q^{j-1}} L_{j-1}^{*} }
\tau ^{j} +\sum _{j=2}^{\infty }
\frac{ (-1)^{j} T^{(\sigma -1)q^{j-2}} }{ L_{j-2}^{*} } \tau ^{j} = 0 .
\end{equation*}
In the same fashion, we apply the expression of $\Psi _{y}$ to obtain
\begin{equation*}
\begin{aligned}
&\sum _{j=2}^{\infty }\frac{ (-1)^{j}(y^{q^{j}} - y) }{ L_{j}^{*} }
\tau ^{j} - \sum _{j=2}^{\infty }
\frac{ (-1)^{j} (y^{q^{j-1}} + \zeta ^{q^{j-1}}T^{(\sigma +q)q^{j-1}}) }{ T^{q^{j-1}} L_{j-1}^{*} }
\tau ^{j}\\
&\quad  +\sum _{j=2}^{\infty }
\frac{ (-1)^{j} \zeta ^{q^{j-2}} T^{(\sigma -1)q^{j-2}} }{ L_{j-2}^{*} }
\tau ^{j} = 0 .
\end{aligned}
\end{equation*}
Combining the $ \tau ^{j} $-terms of the two equations above, we have
\begin{align*}
\frac{ (y^{q^{j}} - y - \zeta ^{q^{j}} (x^{q^{j}} - x)) }{ L_{j}^{*} }
=
\frac{ (y^{q^{j-1}} + \zeta ^{q^{j-1}}T^{(\sigma +q)q^{j-1}}) - \zeta ^{q^{j}} \cdot (x^{q^{j-1}}+T^{(\sigma +q)q^{j-1}}) }{ T^{q^{j-1}} L_{j-1}^{*} }.
\end{align*}
This formula reduces to
%
\begin{align}
\label{Eq:LjLj-1}
L_{j}^{*} = L_{j-1}^{*}
\frac{t-t^{q^{j}}}{ (t - \zeta ^{q^{j+1}})(t-\zeta )^{{q^{j}}-q^{j-1}} (t - \zeta ^{q})^{q^{j-1}} },
\end{align}
by noticing that
\begin{align*}
y^{q^{j}} - y - \zeta ^{q^{j}} (x^{q^{j}} - x) =
\frac{t-t^{q^{j}}}{(t - \zeta ^{q})^ {q^{j}} (t - \zeta ^{q^{j+1}})},
\end{align*}
and
\begin{align*}
\frac{ (y^{q^{j-1}} + \zeta ^{q^{j-1}}T^{(\sigma +q)q^{j-1}}) - \zeta ^{q^{j}} \cdot (x^{q^{j-1}}+T^{(\sigma +q)q^{j-1}} ) }{ T^{q^{j-1}}}
& = \Big(\frac{t- \zeta}{t- \zeta ^{q}}\Big)^{q^{j} - q^{j-1}}.
\end{align*}
The proof is completed since the definition of $ L_{j} $ is exactly composed
of \text{\eqref{Eq:L1L0}} and \text{\eqref{Eq:LjLj-1}}.
\end{proof}

\subsection{Exponential function of standard lattice}
\label{sec3.3}

Notice that $ \mathcal{A}\subseteq \mathbb{C}_{\zeta ^{q}}$ is a canonical
rank one $\mathcal{A}$-lattice with generator $1_{\mathcal{A}}$. We wish
to describe the correspondent exponential function in this section.
%
\begin{defn}
\label{defn3.8}
Recall $ \mathcal{A}= \bigcup _{k} \mathbb{L}(k P_{\pi}) $ in Equation \text{\eqref{Eq:A=Lspan}}, where $\mathbb{L}(k P_{\pi})$ denotes the Riemann-Roch
space. Let $E_{k} (X)$ be the polynomial
\begin{equation*}
E_{k} (X) := X \prod _{ 0 \not = a \in \mathbb L(k P_{\pi})} (1-a^{-1}X)
.
\end{equation*}
The exponential function associated to the standard lattice
$ \mathcal A $ is given by
\begin{equation*}
E_{\infty} (X) := \lim _{k \to \infty } E_{k}(X) = X \prod _{a \in
\mathcal{A}\setminus \{ 0 \} } (1-a^{-1}X).
\end{equation*}
\end{defn}
This definition is valid since the zero locus of $E_{\infty} (X) $ equals
$\mathcal{A}$.

\begin{notation}
For $r=0,1,2,\dots $, $z\in \mathbb{C}_{\zeta ^{q}}$, define
$ \left [ \frac{z}{r} \right ] $ to be the coefficients of the expression:
%
\begin{align}
\label{Eq:explog}
\exp _{\Lambda} ( z \log _{\Lambda} (X) ) = \sum _{r=0}^{\infty }
\left [ \frac{z}{r} \right ] X^{q^{r}}.
\end{align}
\end{notation}
Applying \text{Proposition~\ref{Prop:explog}} and equating coefficients of
$X^{q^{r}}$ in both sides of \text{\eqref{Eq:explog}}, we have
\begin{equation*}
\left [ \frac{z}{r} \right ] = \sum _{j=0}^{r} (-1)^{r-j}
\frac{z^{q^{j}} }{ D_{j} L_{r-j}^{q^{j}}}.
\end{equation*}
This gives a nice alternative expression of $ \Psi $, namely,
%
\begin{align}
\label{Eq:Psiaexp}
\Psi _{a} = \exp _{\Lambda} ( a \log _{\Lambda} (X) ) = \sum _{j=0}^{
\deg (a)} \left [ \frac{a}{j} \right ] \tau ^{j}
\end{align}
for each $ a\in \mathcal A$. Notice that $D_{0}=L_{0}=1$,
$ D_{1} = L_{1} $, so the first few terms are given as follows:
\begin{align*}
\left [ \frac{z}{0} \right ] &= z;
\\
\left [ \frac{z}{1} \right ] &= - \frac{z}{L_{1} } +
\frac{z^{q}}{ D_{1} } = \frac{z^{q}- z }{ L_{1} };
\\
\left [ \frac{z}{2} \right ] &= \frac{z}{ L_{2} } -
\frac{z^{q}}{ L_{1}^{q+1} } + \frac{z^{q^{2}}}{ D_{2} };
\\
\left [ \frac{z}{3} \right ] &= -\frac{z}{ L_{3} } +
\frac{z^{q}}{D_{1} L_{2}^{q}} -
\frac{z^{q^{2}}}{ D_{2} L_{1}^{q^{2}} } + \frac{z^{q^{3}}}{ D_{3} }.
\end{align*}
The Riemann-Roch space $\mathbb L(k P_{\pi})$ is spanned by the elements
\begin{equation*}
\{ x^{i}, x^{i-1} y \}_{i\leqslant k }.
\end{equation*}
This implies that each $ f \in \mathbb L( k P_{\pi })$ is of degree
$ \leqslant 2k $. Applying Equation \text{\eqref{Eq:Psiaexp}}, we have
\begin{equation*}
\Psi _{f} = \sum _{j=0}^{2 k} \left [ \frac{f}{j} \right ] \tau ^{j} ,
\end{equation*}
and $ \left [ \frac{f}{j} \right ] = 0 $ for each $ j > 2k $. In particular,
$ \left [ \frac{f}{2 k +1 } \right ] = 0 $, so we conclude that
$\mathbb L(k P_{\pi}) $ is contained in the zero locus of the polynomial
$ \left [ \frac{z}{2k +1 } \right ] $. In addition,
\begin{equation*}
\mathrm{dim}_{\mathbb{F}_{q}} \mathbb L(k P_{\pi}) = 2k +1 = \deg
\left (\left [ \frac{z}{2k +1 } \right ]\right ).
\end{equation*}
Therefore, the roots of the polynomial
$ \left [ \frac{z}{2k +1 } \right ] $ are exactly the functions contained
in $\mathbb L( k P_{\pi })$. Thus, we obtain the following result.
%
\begin{prop}%
\label{prop:Ek}
The polynomial $E_{k}(z)$ verifies
\begin{equation*}
E_{k}(z) = - L_{2 k +1 } \cdot \left [ \frac{z}{2 k +1 } \right ] =
\sum _{j=0}^{2k+1} (-1)^{j}
\frac{L_{2 k + 1}}{D_{j} L_{2k+1- j} ^{q^{j}}} z^{q^{j}}.
\end{equation*}
Therefore,
%
\begin{equation}
\label{Eq:E_infz}
E_{\infty} (z) = \lim _{ k\to \infty} \sum _{j=0}^{2k+1} (-1)^{j}
\frac{L_{2 k + 1}}{D_{j} L_{2k+1- j}^{q^{j}}} z^{q^{j}}.
\end{equation}
\end{prop}
%

\subsection{The period lattice of standard Drinfeld modules}
\label{sec3.4}

Now we are in the position to compute the Carlitz period. Notice that the
zero locus of $\exp _{\Lambda}$ is the lattice
$ \Lambda = \xi _{\mathcal{A}}\cdot \mathcal{A}$, we find that
\begin{equation*}
E_{\infty}(X) = \frac{1}{\xi _{\mathcal{A}}}\exp _{\Lambda} ( \xi _{
\mathcal{A}}X).
\end{equation*}
Combining this with the equalities \text{\eqref{Eq:expLambda}} and \text{\eqref{Eq:E_infz}} implies that
\begin{align*}
\sum _{j=0}^{\infty }
\frac{\xi _{\mathcal{A}}^{q^{j}} X^{q^{j}}}{D_{j}} & = \exp _{\Lambda}(
\xi _{\mathcal{A}}X )
\\
& = \xi _{\mathcal{A}}\sum _{j=0}^{\infty} \lim _{k\to \infty}
\frac{L_{2 k + 1}}{(-1)^{j} L_{2k+1- j} ^{q^{j}}}
\frac{X^{q^{j}}}{D_{j}}.
\end{align*}
Therefore, we have for each $ j \geqslant 1 $ that
\begin{equation*}
\xi _{\mathcal{A}}^{q^{j}-1} = \lim _{k\to \infty}
\frac{L_{2 k + 1}}{(-1)^{j} L_{2k+1- j} ^{q^{j}}}.
\end{equation*}
In particular, $\xi _{\mathcal{A}}$ is a $ (q-1)$-th root of
$ \lim \limits _{k\to \infty} \frac{-L_{2 k + 1}}{L_{2k} ^{q}} $. It reduces
to compute the limit in terms of infinite product.
%
\begin{notation}
Let us define the partial products
\begin{equation*}
\alpha _{d}:= \prod _{j=2}^{d} \left (1 - \frac{ [2(j-1)]}{[2j]}
\right )
\end{equation*}
and
\begin{equation*}
\beta _{d} := \prod _{j=1}^{d} \left ( 1 - \frac{ [2j-1]}{[2j+1]}
\right ) .
\end{equation*}
\end{notation}

\begin{lem}%
\label{Lem:albed}
Both $ \alpha _{d} $ and $ \beta _{d} $ converge when $d$ tends to infinity.
In addition, the following equalities hold:
\begin{equation*}
\alpha _{d} =
\frac{[2]^{\frac{q^{2d} - 1}{q^{2}-1}}} {\prod _{j=1}^{d} [2j]}
\quad \text{and} \quad \beta _{d} =
\frac{ [1] [2]^{\frac{q^{2d+1} - q}{q^{2}-1}} } {\prod _{j=0}^{d} [2j+1] }.
\end{equation*}
\end{lem}
\begin{proof}
It follows from definition that
\begin{equation*}
v_{\frac{1}{T}}([k])=-q^{k};
\end{equation*}
thus we have
\begin{equation*}
v_{\frac{1}{T}}\Big(\frac{ [2(j-1)]}{[2j]}\Big)=q^{2j}-q^{2(j-1)}
\end{equation*}
and
\begin{equation*}
v_{\frac{1}{T}}\Big(\frac{ [2j-1]}{[2j+1]}\Big)=q^{2j+1}-q^{2j-1}.
\end{equation*}
These valuations guarantee the convergences of $ \alpha _{d} $ and
$\beta _{d} $, respectively. Notice that
\begin{equation*}
[2(k+1)] - [2k] = (T^{q^2} - T)^{q^{2k}} = [2]^{q^{2k}},
\end{equation*}
and
\begin{equation*}
[2k+1] - [2k-1] =( T^{q^2} - T)^{q^{2k-1}} = [2]^{q^{2k-1}}.
\end{equation*}
Then $\alpha _{d}$ can be rewritten as
\begin{equation*}
\alpha _{d}=\frac{[2]^{q^{2}}}{[4]}\frac{[2]^{q^{4}}}{[6]} \cdots
\frac{[2]^{q^{2(d-1)}}}{[2d]} =
\frac{[2]^{\frac{q^{2d} - 1}{q^{2}-1}} } {\prod _{j=1}^{d} [2j]}.
\end{equation*}
Similarly, $\beta _{d}$ can be rewritten as
\begin{equation*}
\beta _{d}=\frac{[2]^{q}}{[3]}\frac{[2]^{q^{2}}}{[5]} \cdots
\frac{[2]^{q^{2d-1}}}{[2d+1]} =
\frac{[1][2]^{\frac{q^{2d+1} - q}{q^{2}-1}} } {\prod _{j=0}^{d} [2j+1]}.\qedhere
\end{equation*}
\end{proof}
%
\begin{thm}%
\label{Thm:xiA}
The Carlitz period $\xi _{\mathcal{A}}$ is given by
\begin{equation*}
\xi _{\mathcal{A}}= \frac{\sqrt[q-1]{-[2]} T^{\sigma -1}}{[1]} \cdot
\prod _{i=1}^{\infty }\left (1 - \frac{[i]}{[i+2]} \right ) ,
\end{equation*}
where $ \sqrt[q-1]{-[2]}$ denotes a $(q-1)$-th root of $-[2]$.
\end{thm}
\begin{proof}
From \text{Lemmas~\ref{Lem:leLj} and \ref{Lem:albed}}, we have
\begin{equation*}
\alpha _{d}\beta _{d}=
\frac{[1][2]^{\frac{q^{2d+1}+q^{2d}-q - 1}{q^{2}-1}}T^{(q^{2d+1}-1)(\sigma -1)}}{L_{2d+1}},
\end{equation*}
and
\begin{equation*}
\alpha _{d}\beta _{d-1}=
\frac{[1][2]^{\frac{q^{2d}+q^{2d-1} -q- 1}{q^{2}-1}}T^{(q^{2d}-1)(\sigma -1)}}{L_{2d}}.
\end{equation*}
Thus, we have
\begin{equation*}
-\frac{L_{2d+1}}{L_{2d}^{q}}=-[1]^{1-q}[2] \cdot T^{(q-1)(\sigma -1)}
\alpha _{d}^{q-1}\beta _{d-1}^{q}\beta _{d}^{-1}.
\end{equation*}
Taking the limit $d\to \infty $, we have
\begin{equation*}
\xi _{\mathcal{A}}^{q-1} = \lim _{d\to \infty } -
\frac{L_{2d+1}}{L_{2d}^{q}}=-[1]^{1-q}[2] \cdot T^{(q-1)(\sigma -1)}
\alpha _{\infty}^{q-1}\beta _{\infty}^{q-1}.
\end{equation*}
Then the lemma follows by taking $(q-1)$-th root.
\end{proof}
In order to represent the lattice $\Lambda ^{\sigma }$ of
$ \Psi ^{\sigma }$, we need to construct the isogeny between the two standard Drinfeld modules.
\begin{lem}\label{Lem:Phiisg}
Let  $ \ell $ denote a $(q-1)$-th root of $ T^{\sigma-1} $. Then the twist polynomial 
\[ \lambda=\ell(\tau+T) = \ell \Psi_{I_\infty} \] 
defines an isogeny from $ \Psi $ to  $ \Psi^{\sigma} $. 
\end{lem}
\begin{proof}
  According to Definition \ref{Defn:isogeny},   $ \lambda=\ell(\tau+T) $ is an isogeny from  Drinfeld module $ \Psi $ to  $ \Psi^{\sigma} $  if and only if the equality
\begin{align}\label{Eq:isgPhi}
   \lambda \Psi_{a} = \Psi^{\sigma}_{a} \lambda   
\end{align}
holds for  $ a \in \mathcal{A} $. Since $ \mathcal{A}$ is generated by $x$ and $y$, we shall only check  the equality \eqref{Eq:isgPhi} for $a =x $ and $a= y$. 
Substituting the expressions of both $\Psi $ and $\Psi^\sigma$ into \eqref{Eq:isgPhi}, we obtain the following relations:
\begin{equation}\label{Eq:isgPhix}
  \ell(\tau+T) T^{\sigma-1}  (\tau+T)=T^{1-\sigma}  (\tau+T^{\sigma})^2\ell
  \end{equation}
  and
  \begin{equation}\label{Eq:isgPhiy}
  \ell(\tau+T)\zeta T^{\sigma-1}  (\tau+\frac{t}{\zeta}T)=\zeta^q  T^{1-\sigma} (\tau+\frac{t}{\zeta^q}T)(\tau+T^{\sigma})\ell.
\end{equation} 
 By comparing the coefficients on both sides of Equations \eqref{Eq:isgPhix} and \eqref{Eq:isgPhiy} respectively, we derive the following conditions: 
    \[
    \ell^{q^2-1}=T^{(q+1)(\sigma-1)};
    \] 
\begin{align}
\label{Eq:coell}
&\ell\big(T^{q\sigma}+T^{\sigma}\big)=\big(T+T^{(q-1)\sigma+1}\big)\ell^q.
\end{align}
Both conditions above are equivalent to 
\[
	\ell^{q-1}=T^{\sigma-1}.
\] 
This completes the proof.
\end{proof}

The following theorem (see \cite[Corollary 4.9.5]{Goss96} or \cite[Theorem 5.2.11 ]{P23}) establishes a connection between two lattices of isogeny Drinfeld modules.
\begin{thm}\label{Thm:isgdm}
    Let $\phi$ and $\psi$ be two Drinfeld modules over $\mathbb {C}_{\infty}$ associated with lattices $\Lambda_{\phi}$ and $\Lambda_{\psi}$ respectively. Suppose that the twist polynomial $\lambda$ is an isogeny from 
$\phi$ to $\psi$ with nonzero constant term $ \partial(\lambda) $ and $I$ is the ideal in $\mathcal{A}$ annihilated by $\lambda$, i.e., $ u = \ell \phi_I$ for some constant $\ell$. Then we have
\[
	\Lambda_{\psi} =  \Lambda_{\phi} \cdot \partial(\lambda) I^{-1}. 
\]
\end{thm}

Theorem \ref{Thm:isgdm} together with Lemma \ref{Lem:Phiisg} gives rise to the expression of the lattice $\Lambda^{\sigma}$.  In summary, we have the following result.
\begin{thm}
Let $\Psi$ and $\Psi^{\sigma}$ be the two standard models of rank one Drinfeld modules and let $\Lambda$ and $ \Lambda^{\sigma}$ be the associated lattices respectively. Then, we find that $\Lambda=\xi_{\mathcal{A}} \cdot \mathcal{A}$ and  
    $ \Lambda^{\sigma} = T \ell\xi_{\mathcal{A}} \cdot I_{\infty}^{-1}$, where $\xi_{\mathcal{A}} $ is given in Theorem \ref{Thm:xiA}, and  $ \ell $ is a $(q-1)$-th root of $ T^{\sigma-1} $.
\end{thm}

\section{Rank two Drinfeld modules}
\label{Sec:RTDM}

In this section we compute the complete family of Drinfeld modules of rank
two. For simplicity, we focus our attention on the case of normalized
$\zeta ^{q} $-type, i.e., the rank two Drinfeld module $ \phi $ verifying
$ \operatorname{LT}_{\phi}(a) = \operatorname{Sgn}_{
\zeta ^{q}}(a) $ for any $ a\in \mathcal{A}$.

\subsection{Rank two normalized Drinfeld modules}
\label{sec4.1}

In this section, we only consider Drinfeld modules over
$\mathcal{A}$-field $ \bar{K } $.
%
\begin{notation}%
\label{No:S4phiItphiI}
Recall $I_{0}$ and $I_{\infty}$ defined in Section~\ref{Sec:rodm}. Denote
by $ \phi _{I_{0}}$ and $ \phi _{I_{\infty}}$ the annihilators of
$I_{0}$ and $I_{\infty}$, respectively. The twist polynomial
$ \phi _{I_{\infty}}$ induces isogeny from $ \phi $ to another one
$ \tilde \phi $. Analogously, we assume that $\tilde \phi _{I_{0}}$ and
$\tilde \phi _{I_{\infty}}$ are the annihilators of $\tilde I_{0}$ and
$\tilde I_{\infty}$ under $ \tilde\phi $-action, respectively.
\end{notation}

As we are considering the rank two Drinfeld modules, the constructed annihilators
$ \phi _{I_{0}}$, $ \phi _{I_{\infty}}$, $\tilde \phi _{I_{0}}$ and
$\tilde \phi _{I_{\infty}}$ are monic twist polynomials of degree two.
We emphasize that these twist polynomials enable us to recover both
$ \phi $ and $ \tilde\phi $.

\begin{lem}
\label{Lem:phi}
An arbitrary rank two Drinfeld module $ \phi $ over $ \mathcal{A}$-field
$ \bar{K} $ verifying $ \operatorname{LT}_{\phi}(a) =
\operatorname{Sgn}_{\zeta ^{q}}(a) $ for any $a \in \mathcal{A}$ can be
represented by
%
\begin{align}
\label{Eq:lephi}
\begin{cases}
\phi _{x} = \tilde\phi _{I_{\infty}} \phi _{I_{\infty}}
\\
\phi _{y} = \zeta ^{q} \tilde\phi _{I_{0}} \phi _{I_{\infty}}.
\end{cases}
\end{align}
Moreover, $ \tilde{\phi}_{I_{\infty}} $ induces an inverse isogeny from
$ \tilde{\phi}$ to $ \phi $, and $ \tilde{\phi}$ is of the form
%
\begin{align}
\label{Eq:letphi}
\begin{cases}
\tilde\phi _{x} = \phi _{I_{\infty}} \tilde\phi _{I_{\infty}}
\\
\tilde\phi _{y} = \zeta ^{q} \phi _{I_{0}} \tilde \phi _{I_{\infty}}.
\end{cases}
\end{align}
\end{lem}

\begin{proof}
Since $ \phi _{I_{\infty}} $ gives isogeny from $\phi $ to
$ \tilde\phi $, we obtain the equality
%
\begin{align}
\label{Eq:phiinfis}
\phi _{I_{\infty}} \phi _{a} = \tilde\phi _{a} \phi _{I_{\infty}}
\end{align}
for $ a \in \mathcal A $. Since $ I_{\infty }= (x,y) $, both
$\phi _{x}$ and $ \phi _{y} $ are right-divisible by
$ \phi _{I_{\infty}}$, and accordingly both $\tilde \phi _{x}$ and
$ \tilde\phi _{y} $ are right-divisible by
$ \tilde {\phi}_{I_{\infty}}$. Thus we can set
%
\begin{equation}
\label{Eq:phixU1}
\phi _{x}=U_{1}\phi _{I_{\infty}}
\end{equation}
for some twist polynomial $ U_{1} $. We assert that $ U_{1} $ is the annihilator
of $ (x,y) $ with $ \tilde \phi $-action, i.e.,
$ U_{1} = \tilde {\phi}_{I_{\infty }}$ is the common right-divisor of both
$ \tilde\phi _{x} $ and $ \tilde\phi _{y} $. From Equation \text{\eqref{Eq:phiinfis}}, we know that $U_{1}$ is a right-divisor of
$\tilde \phi _{x}$. On the other hand, as $\phi _{xy}$ and
$\phi _{y^{2}} = \phi _{ \mathrm{Tr}(\zeta ) x y - \zeta ^{q+1} x^{2} +
x}$ are right-divisible by $\phi _{x}$, and
$\phi _{I_{\infty}}\phi _{y}$ is the annihilator of the ideal
$(xy, y^{2})$, we derive that $\phi _{x} $ is a right-divisor of
$\phi _{I_{\infty}}\phi _{y} $. Notice that
$\phi _{x} = U_{1} \phi _{I_{\infty}}$ and
$\phi _{I_{\infty}}\phi _{y} = \tilde\phi _{y} \phi _{I_{\infty}} $. We
can cancel $ \phi _{I_{\infty}}$ of both sides to obtain that
$ U_{1} $ is a right-divisor of $ \tilde\phi _{y} $. Substituting
$U_{1} =\tilde{\phi}_{I_{\infty}} $ in Equation \text{\eqref{Eq:phixU1}}, we have
%
\begin{align}
\label{Eq:phix}
\phi _{x}=\tilde\phi _{I_{\infty}}\phi _{I_{\infty}}.
\end{align}

Similarly, notice that $ I_{0} = (x-\zeta ^{-q-1},y)$, we can set
$\phi _{y}=\zeta ^{q} U_{2}\phi _{I_{\infty}} $ for some rank two twist
polynomial $ U_{2} $. It requires to show that
$ U_{2} = \tilde \phi _{I_{0}}$, and thus
\begin{align*}
\phi _{y} = \zeta ^{q} \tilde \phi _{I_{0}} \phi _{I_{\infty}}.
\end{align*}
In the same fashion, we deduce from Equation \text{\eqref{Eq:phiinfis}} that
$U_{2}$ is right-divisible by $\tilde \phi _{y}$. As both
$\phi _{y(x-\zeta ^{-q-1})}$ and
$\phi _{x(x-\zeta ^{-q-1})}= \phi _{\zeta ^{-q-1} \mathrm{Tr}(\zeta ) x
y - \zeta ^{-q-1} y^{2} }$ are right-divisible by $\phi _{y}$, and
$\phi _{I_{\infty}}\phi _{x-\zeta ^{-q-1}}$ is the annihilator of
$\big(y(x-\zeta ^{-q-1}), ~x(x-\zeta ^{-q-1})\big)$, it yields that
$\phi _{y}$ is a right-divisor of
$\phi _{I_{\infty}}\phi _{x-\zeta ^{-q-1}}$. Substituting
$ x - \zeta ^{-q-1}$ in Equation \text{\eqref{Eq:phiinfis}} yields
\begin{equation*}
\phi _{I_{\infty}} \phi _{x - \zeta ^{-q-1}} = \tilde\phi _{x -
\zeta ^{-q-1}} \phi _{I_{\infty}}.
\end{equation*}
Therefore, $\tilde\phi _{x - \zeta ^{-q-1}} \phi _{I_{\infty}} $ is right-divisible
by $\phi _{y}=\zeta ^{q} U_{2}\phi _{I_{\infty}} $, which implies that
$ U_{2} $ is a right-divisor of $ \tilde\phi _{x-\zeta ^{-q-1}}$.

In combination with Equations \text{\eqref{Eq:phiinfis}} and \text{\eqref{Eq:phix}},
we have
\begin{align*}
\tilde \phi _{I_{\infty}}\tilde \phi _{a}\phi _{I_{\infty}} &=
\tilde{\phi}_{I_{\infty}} \phi _{I_{\infty}}\phi _{a}
\\
& =\phi _{x}\phi _{a}= \phi _{a}\phi _{x}
\\
& = \phi _{a} \tilde\phi _{I_{\infty}}\phi _{I_{\infty}}
\end{align*}
for any $a \in \mathcal{A}$. By canceling $\phi _{I_{\infty}}$, we have
$\tilde \phi _{I_{\infty}}\tilde \phi _{a} = \phi _{a} \tilde\phi _{I_{
\infty}} $. It means that $\tilde \phi _{I_{\infty}}$ represents an isogeny
from $ \tilde \phi $ to $ \phi $. Therefore, the expression in Equation \text{\eqref{Eq:letphi}} can be obtained by the previous argument.
\end{proof}

According to the expressions in Equations \text{\eqref{Eq:lephi}} and \text{\eqref{Eq:letphi}}, we can further assume that
%
\begin{align}
\phi _{I_{\infty }} &= \tau ^{2} + \alpha \tau + \delta ,
\label{Eq:phiiinf}
\\
\phi _{I_{0} } &= \tau ^{2} + \beta \tau +
\frac{y}{\zeta ^{q} \tilde\delta},
\label{Eq:phii0}
\\
\tilde\phi _{I_{\infty }} &= \tau ^{2} + \tilde{\alpha} \tau +
\frac{x}{\delta},
\label{Eq:tphiiinf}
\\
\tilde \phi _{I_{0} } &= \tau ^{2} + \tilde{\beta} \tau +
\frac{y}{\zeta ^{q}\delta},
\label{Eq:tphii0}
\end{align}
where $\alpha $, $\beta $, $\delta $, $\tilde\alpha $,
$\tilde\beta $, $\tilde\delta $ are indeterminates. The following lemma
yields that the motive structure of $ \phi $ can be represented by these
indeterminates.

\begin{lem}%
\label{Lem:phixyABC}
With the notations above, the motive structure of $ \phi $ is represented
by
\begin{equation*}
(y \otimes 1 - x \otimes \zeta -1 \otimes C \alpha )\cdot \tau + (y
\otimes A - x \otimes B -1 \otimes C \delta ) \cdot 1_{\phi }= C
\tau ^{2},
\end{equation*}
with
\begin{equation*}
A = \frac{\zeta }{\zeta - \zeta ^{q}}\lambda ^{q}, \qquad B =
\frac{\zeta ^{2} }{\zeta - \zeta ^{q}}\lambda ^{q}, \qquad C =
\frac{T }{(1-\zeta ^{q-1})} \frac{\lambda ^{q}}{\delta},
\end{equation*}
and $ \lambda = \tilde\beta - \tilde \alpha $. In addition, the wedge product
$ \wedge ^{2} \phi $ is given by $ \psi ^{(\zeta ,-\frac{A}{C})} $ (see
\text{Definition~\ref{Defn:psize*}}).
\end{lem}
\begin{proof}
Substituting Equations \text{\eqref{Eq:tphiiinf}} and \text{\eqref{Eq:tphii0}} in Equations \text{\eqref{Eq:lephi}}, it yields
%
\begin{align}
\label{Eq:phixy3}
\begin{cases}
\phi _{x} = (\tau ^{2} + \tilde{\alpha} \tau + \frac{x}{\delta})
\phi _{I_{\infty}}
\\
\phi _{y} = \zeta ^{q} (\tau ^{2} + \tilde{\beta} \tau +
\frac{y}{\zeta ^{q} \delta}) \phi _{I_{\infty}} .
\end{cases}
\end{align}
From the settings $ \lambda = \tilde{\beta}- \tilde{\alpha} $,
$T^{\sigma} = y-\zeta ^{q} x $, and Equations \text{\eqref{Eq:phixy3}}, we have
the equalities
%
\begin{align}
\label{Eq:phixy3t1}
\phi _{y} - \zeta ^{q} \phi _{x} = \big( \lambda \zeta ^{q} \tau +
\frac{T^{\sigma}}{\delta} \big) \phi _{I_{\infty}}
\end{align}
and
%
\begin{align}
\label{Eq:phixy4}
\tau \phi _{y} - \zeta \tau \phi _{x} = \Big( \zeta \lambda ^{q}
\tau ^{2}+\frac{T^{\sigma q}\tau}{\delta ^{q}} \Big ) \phi _{I_{
\infty}} .
\end{align}
By subtracting $\frac{\zeta \lambda ^{q}}{\zeta ^{q} } \phi _{y} $ in both
sides of Equation \text{\eqref{Eq:phixy4}}, it yields
%
\begin{align}
\label{Eq:phixy3t2}
\tau \phi _{y} - \zeta \tau \phi _{x} -
\frac{\zeta \lambda ^{q}}{\zeta ^{q} } \phi _{y} = \Big((
\frac{T^{\sigma q}}{\delta ^{q}} - \tilde{\beta} \zeta \lambda ^{q})
\tau - \frac{\zeta \lambda ^{q} y}{ \zeta ^{q} \delta}\Big) \phi _{I_{
\infty}}.
\end{align}
Using the linear combination
$\text{\text{(\ref{Eq:phixy3t2})}} +
\frac{ \tilde{\beta} \zeta \lambda ^{q} \delta ^{q}- T^{\sigma q }}{\zeta ^{q} \lambda \delta ^{q}}
\text{\text{(\ref{Eq:phixy3t1})}} $, we obtain
%
\begin{align}
& (\tau +A) \phi _{y} - (\zeta \tau + B) \phi _{x}
\nonumber
\\
= & \left ( \tau \phi _{y} - \zeta \tau \phi _{x} -
\frac{\zeta \lambda ^{q}}{\zeta ^{q}} \phi _{y} \right ) +
\frac{ \tilde{\beta} \zeta \lambda ^{q} \delta ^{q}- T^{\sigma q }}{\zeta ^{q} \lambda \delta ^{q}}
\left (\phi _{y} - \zeta ^{q} \phi _{x}\right )
\nonumber
\\
= &\Big(
\frac{\tilde{\beta} \zeta T^{\sigma}\lambda ^{q}\delta ^{q} -T^{\sigma (q+1) }}{\zeta ^{q} \lambda \delta ^{q+1}}
- \frac{\zeta \lambda ^{q} y}{ \zeta ^{q} \delta}\Big) \phi _{I_{
\infty}}
\nonumber
\\
= & C \phi _{I_{\infty}},
\label{Eq:phixyI}
\end{align}
where
%
\begin{align}
A : = &
\frac{ \tilde{\beta} \zeta \lambda ^{q} \delta ^{q}- T^{\sigma q }}{\zeta ^{q} \lambda \delta ^{q}}-
\frac{\zeta \lambda ^{q}}{\zeta ^{q}} = \tilde{\beta} (
\frac{\lambda}{\zeta})^{q-1} -
\frac{ T^{\sigma q }}{\zeta ^{q} \lambda \delta ^{q}} -
\frac{\lambda ^{q}}{\zeta ^{q-1}},
\label{Eq:Af}
\\
B: =&
\frac{ \tilde{\beta} \zeta \lambda ^{q} \delta ^{q}- T^{\sigma q }}{ \lambda \delta ^{q}}
= \tilde{\beta} \zeta \lambda ^{q-1} -
\frac{ T^{\sigma q }}{ \lambda \delta ^{q}},
\label{Eq:Bf}
\\
C := &
\frac{\tilde{\beta} \zeta T^{\sigma}\lambda ^{q} \delta ^{q}-T^{\sigma (q+1) }}{\zeta ^{q} \lambda \delta ^{q+1}}
- \frac{\zeta \lambda ^{q} y}{ \zeta ^{q} \delta} =
\frac{T^{\sigma}}{\zeta ^{q}\delta} \cdot \Big (
\frac{\tilde{\beta} \zeta \lambda ^{q} \delta ^{q} -T^{\sigma q }}{\lambda \delta ^{q}}
- \frac{\zeta \lambda ^{q} t}{ (t -\zeta ^{q} )}\Big).
\label{Eq:Cf}
\end{align}
Substituting the expression of $\phi _{I_{\infty}}$ in Equation \text{\eqref{Eq:phixyI}}, we find that the pair ($ \phi _{x} , \phi _{y}$) verifies
%
\begin{equation}
\label{Eq:phixyre}
(\tau +A) \phi _{y} - (\zeta \tau + B) \phi _{x} = C (\tau ^{2} +
\alpha \tau + \delta ).
\end{equation}
Let $ 1_{\phi }, \tau , \tau ^{2} , \ldots $ be the generators of
$ \bar{M}(\phi ) $. Equation \text{\eqref{Eq:phixyre}} implies that the motive
structure of $ \phi $ verifies
\begin{equation*}
(y \otimes 1 - x \otimes \zeta -1 \otimes C \alpha )\cdot \tau + (y
\otimes A - x \otimes B -1 \otimes C \delta ) \cdot 1_{\phi }= C
\tau ^{2} .
\end{equation*}
Let $ 1_{\psi }= 1 \wedge \tau $ be the generator of $ \psi $. It follows
that
\begin{align*}
C\tau (1_{\psi}) & = C\tau \wedge \tau ^{2}
\\
& = \tau \wedge \left ( (y \otimes 1 - x \otimes \zeta -1 \otimes C
\alpha )\cdot \tau + (y\otimes A - x \otimes B -1 \otimes C \delta )
\cdot 1_{\phi }\right )
\\
& = \left ( 1\otimes C \delta - y \otimes A + x \otimes B \right )( 1
\wedge \tau ) .
\end{align*}
After dividing by $C$, it yields
\begin{equation*}
\tau (1_{\psi }) = \left ( \delta - y \otimes \frac{A}{C} + x
\otimes \frac{B}{C} \right )( 1_{\psi }) .
\end{equation*}
It follows from \text{Proposition~\ref{Prop:psimo}} that
$ \wedge ^{2} \phi = \psi ^{(\zeta ,-\frac{A}{C})} $, and
%
\begin{align}
\label{Eq:ABC}
\frac{A}{C \delta} = \frac{1}{T} , \ \ \ \ \ \frac{B}{C \delta } =
\frac{\zeta}{T}.
\end{align}
Substituting \text{\eqref{Eq:Af}} and \text{\eqref{Eq:Bf}} in the equality
$ B = \zeta A $ shows that
%
\begin{equation}
\label{Eq:betaexp}
\tilde{\beta} = \frac{ T^{\sigma q}}{\zeta \lambda ^{q}\delta ^{q}}+
\frac{\lambda }{1-\zeta ^{q-1}}.
\end{equation}
Noting that $\lambda =\tilde \beta -\tilde \alpha $, we have
%
\begin{equation}
\label{Eq:alphaexp}
\tilde{\alpha} = \frac{ T^{\sigma q }}{\zeta \lambda ^{q}\delta ^{q}}+
\frac{\zeta ^{q}\lambda }{\zeta -\zeta ^{q}}.
\end{equation}
Furthermore, substituting $\tilde{\alpha}$ and $ \tilde{\beta}$ in Equations \text{\eqref{Eq:Af}}-\text{\eqref{Eq:Cf}}, we have
\begin{equation*}
A = \frac{\zeta }{\zeta - \zeta ^{q}}\lambda ^{q}, \qquad B =
\frac{\zeta ^{2} }{\zeta - \zeta ^{q}}\lambda ^{q}
\end{equation*}
and
%
\begin{align}
\label{Eq:C}
C &= \frac{BT}{\zeta \delta} = \frac{T }{(1-\zeta ^{q-1})}
\frac{\lambda ^{q}}{\delta}.\qedhere
\end{align}
\end{proof}

Now we are in a position to prove our main result.
%
\begin{thm}%
\label{Thm:thphi}
Let us fix $ \nu $ as a $(q+1)$-th root of $ -T^{\sigma +q} $. An arbitrary
$\zeta ^{q}$-type rank two normalized Drinfeld module over
$ \mathcal{A} $-field $ \bar K $ can be represented by
$(\phi _{x}, \phi _{y} )$ of the forms
\begin{equation*}
\begin{cases}
\phi _{x} := (\tau ^{2} + \tilde{\alpha} \tau +
\frac{x}{\nu \lambda ^{q-1}} ) (\tau ^{2}+ \alpha \tau + \nu \lambda ^{q-1})
\\
\phi _{y} := \zeta ^{q} (\tau ^{2} + \tilde{\beta} \tau +
\frac{y}{\zeta ^{q} \nu \lambda ^{q-1}} ) (\tau ^{2} + \alpha \tau +
\nu \lambda ^{q-1}) ,
\end{cases}
\end{equation*}
with
\begin{align*}
\tilde{\alpha} & = -
\frac{ \nu T^{\sigma q-(q+\sigma )}}{\zeta \lambda ^{q^{2}}} +
\frac{\zeta ^{q}\lambda }{\zeta -\zeta ^{q}},
\\
\tilde{\beta} & =-
\frac{ \nu T^{\sigma q-(q+\sigma )}}{\zeta \lambda ^{q^{2}}}+
\frac{\zeta \lambda }{\zeta -\zeta ^{q}} ,
\\
\alpha & =\frac{\lambda ^{q^{2}}}{1-\zeta ^{1-q}}+
\frac{\nu }{\zeta T\lambda }.
\end{align*}
Moreover, the wedge product  $ \wedge ^{2} \phi $ is identical to $\psi^{(\zeta,-\frac{\nu\lambda^{q-1}}{T})} $
{$\psi ^{(\zeta ,-\frac{\delta}{T})} $} defined in Equation \text{\eqref{Eq:defpsi}}, namely,
\begin{equation*}
\begin{cases}
\psi_x =  -\lambda^{1-q^2}\big(\tau-(\nu\lambda^{q-1})^q T^{1-q}\big)(\tau - \nu\lambda^{q-1})  \\
			\psi_y = -\zeta\lambda^{1-q^2}\big(\tau-\frac{t(\nu\lambda^{q-1})^q T^{1-q}}{\zeta}\big) (\tau -\nu\lambda^{q-1}). 
\end{cases}
\end{equation*}
\end{thm}

\begin{proof}
Adding $(\zeta \tau + B)\zeta ^{-q-1}$ to both sides of Equation \text{\eqref{Eq:phixyre}}, it follows
%
\begin{equation}
\label{Eq:tauABphiInf}
(\tau +A) \phi _{y} + (\zeta \tau + B) (\zeta ^{-q-1}- \phi _{x}) = C (
\phi _{I_{\infty}} + \frac{1}{\zeta ^{q} C} \tau +
\frac{\delta}{\zeta ^{q} T}).
\end{equation}
Recalling that $ \phi _{I_{0}}$ is an annihilator of
$I_{0} = (\zeta ^{-1-q} - x , y ) $, we know that $\phi _{I_{0}}$ is a
monic polynomial of $\tau $-degree two. Now both sides of \text{\eqref{Eq:tauABphiInf}} are divisible by $ \phi _{I_{0}}$, and thus
\begin{align*}
\phi _{I_{0}} = \phi _{I_{\infty}} + \frac{1}{\zeta ^{q} C} \tau +
\frac{\delta}{\zeta ^{q} T} = \tau ^{2} + (\alpha +
\frac{1}{\zeta ^{q} C}) \tau + \delta (1+\frac{1}{\zeta ^{q}T}).
\end{align*}
It follows from the expansion of $C$ in Equation \text{\eqref{Eq:C}} that
%
\begin{align}
\label{Eq:phiI0i1}
\phi _{I_{0}} - \phi _{I_{\infty }} &= \frac{1}{\zeta ^{q} C} \tau +
\frac{\delta}{\zeta ^{q} T} =\frac{(1-\zeta ^{q-1})}{T \zeta ^{q} }
\frac{\delta} {\lambda ^{q}} \tau + \frac{\delta}{\zeta ^{q} T}.
\end{align}
From Equations \text{\eqref{Eq:tphiiinf}} and \text{\eqref{Eq:tphii0}}, we have
%
\begin{align}
\label{Eq:tphiI0i1}
\tilde \phi _{I_{0}} - \tilde \phi _{I_{\infty }} = (\tilde{\beta} -
\tilde{\alpha}) \tau + \frac{y}{\zeta ^{q} \delta} -
\frac{x}{\delta } = \lambda \tau +
\frac{T^{\sigma }}{\zeta ^{q} \delta } .
\end{align}
Substituting Equations \text{\eqref{Eq:phiiinf}}-\text{\eqref{Eq:tphiiinf}} in Equations \text{\eqref{Eq:letphi}}, we have
\begin{equation*}
\begin{cases}
\tilde \phi _{x} = (\tau ^{2} + \alpha \tau + \frac{x}{\tilde\delta} )
(\tau ^{2} + \tilde{\alpha} \tau + \tilde{\delta} )
\\
\tilde \phi _{y} = \zeta ^{q} (\tau ^{2} + \beta \tau +
\frac{y}{\zeta ^{q} \tilde\delta} )(\tau ^{2} + \tilde{\alpha} \tau +
\tilde{\delta} ).
\end{cases}
\end{equation*}
Setting $ \tilde{\lambda} = \beta - \alpha $, we start another process
in parallel that
%
\begin{align}
\label{Eq:tphiI0i2}
\tilde{\phi}_{I_{0}} - \tilde{\phi}_{I_{\infty }} &=
\frac{(1-\zeta ^{q-1})}{T \zeta ^{q} }
\frac{\tilde{\delta}} {\tilde{\lambda}^{q}} \tau +
\frac{\tilde{\delta}}{\zeta ^{q} T},
\\
\label{Eq:phiI0i2}
\phi _{I_{0}} - \phi _{I_{\infty }} & = \tilde{\lambda} \tau +
\frac{T^{\sigma }}{\zeta ^{q} \tilde{\delta}},
\end{align}
which are analogous to Equations \text{\eqref{Eq:phiI0i1}} and \text{\eqref{Eq:tphiI0i1}} respectively. Hence, comparing \text{\eqref{Eq:phiI0i1}} with \text{\eqref{Eq:phiI0i2}} shows that $\tilde{\delta} = \frac{x}{\delta}$ and
%
\begin{align}
\label{Eq:tlade}
\tilde{\lambda} = \frac{(1-\zeta ^{q-1})}{T \zeta ^{q} }
\frac{\delta} {\lambda ^{q}} .
\end{align}
Accordingly, Equations \text{\eqref{Eq:tphiI0i1}} and \text{\eqref{Eq:tphiI0i2}} yield
%
\begin{align}
\label{Eq:la}
\lambda = \frac{(1-\zeta ^{q-1})}{T \zeta ^{q} }
\frac{\tilde\delta} {\tilde\lambda ^{q}} .
\end{align}
Together with Equation \text{\eqref{Eq:tlade}}, this implies
\begin{equation*}
\delta ^{q+1} = - T^{\sigma +q} \lambda ^{q^{2}-1},
\end{equation*}
i.e., $ \delta = \nu \lambda ^{q-1} $. In parallel to
$\tilde\alpha $ in Equation \text{\eqref{Eq:alphaexp}}, we have
\begin{equation*}
\alpha =
\frac{ T^{\sigma q }}{\zeta \tilde{\lambda}^{q} \tilde{\delta}^{q}}+
\frac{\zeta ^{q}\tilde{\lambda} }{\zeta -\zeta ^{q}}.
\end{equation*}
Then the theorem holds by substituting
$ \delta = \nu \lambda ^{q-1}$ into
$\alpha ,~ \tilde{\alpha}, ~\tilde{\beta}$.
\end{proof}

\subsection{Complete family of rank two Drinfeld modules}
\label{sec4.2}

The Drinfeld module $ \phi ^{(\lambda ,\nu )}$ with parameters
$ \lambda ,~ \nu $ occurred in \text{Theorem~\ref{Thm:thphi}} can be expanded
as
%
\begin{align}
\label{Eq:phixexp}
\begin{aligned}
\phi _{x}^{(\lambda ,\nu )} =&\tau ^{4}+\Big(
\frac{\lambda ^{q^{4}}}{1-\zeta ^{1-q}}+
\frac{\zeta ^{q-1}\lambda}{1-\zeta ^{q-1}}\Big)\tau ^{3}+\Big(
\frac{\nu T^{q(q-1)(1-\sigma )-\sigma}\lambda ^{q^{3}-q^{2}}}{\zeta ^{q}-\zeta}-
\frac{T \lambda ^{1-q}}{(\zeta ^{q}-\zeta )\nu }
\\
&+
\frac{T^{\sigma q-q}}{\zeta ^{q+1}\lambda ^{q^{2}+q}}
+\frac{\zeta ^{q-1}\lambda ^{q^{3}+1}}{(1-\zeta ^{q-1})^{2}}\Big)
\tau ^{2}+\Big(\frac{T^{\sigma q}+T^{\sigma}}{\zeta \lambda ^{q}}+
\frac{(T^{\sigma +q}+x)\lambda ^{q^{2}-q+1}}{(1-\zeta ^{1-q})\nu}
\Big)\tau +x,
\end{aligned}
\end{align}
and
%
\begin{align}
\label{Eq:phiyexp}
\begin{aligned}
\phi _{y}^{(\lambda ,\nu )}=&\zeta ^{q}\tau ^{4}+\zeta ^{q}\Big(
\frac{\lambda ^{q^{4}}}{1-\zeta ^{1-q}}+
\frac{\lambda}{1-\zeta ^{q-1}}\Big)\tau ^{3}+\zeta ^{q}\Big(
\frac{\nu T^{q(q-1)(1-\sigma )-\sigma}\lambda ^{q^{3}-q^{2}}}{\zeta ^{q}-\zeta}\\
&-
\frac{T \lambda ^{1-q}}{(\zeta ^{q}-\zeta )\nu }
+\frac{T^{\sigma q-q}}{\zeta ^{q+1}\lambda ^{q^{2}+q}}+
\frac{\lambda ^{q^{3}+1}}{(1-\zeta ^{q-1})^{2}}\Big)\tau ^{2}\\
&+\zeta ^{q}
\Big(
\frac{\zeta ^{q} T^{\sigma q}+t T^{\sigma}}{\zeta ^{q+1}\lambda ^{q}}+
\frac{(\zeta T^{\sigma +q}+y)\lambda ^{q^{2}-q+1}}{(\zeta ^{q}-\zeta )\nu }
\Big)\tau +y.\end{aligned}
\end{align}
%
\begin{defn}
\label{defn4.5}
For a fixed $ \nu $, the family
$ \{ \phi ^{( \lambda , \nu )} \}_{\lambda} $ above is called the complete
family of $ \zeta ^{q} $-type normalized Drinfeld modules of rank two.
\end{defn}

Next we investigate the family of the isomorphism classes of
$ \{ \phi ^{(\lambda ,\nu )} \}_{\lambda}$, which indicates the desired
family as in \text{Theorem~\ref{Thm:cfPhiJ}}. Let
$ \mathbb{F}^{*}: = \mathbb{F}_{q^{4}}^{*} / \mathbb{F}_{q}^{*} $ be a
quotient multiplicative group. We have the conjugate action of
$ \mathbb{F}^{*} $ on the set $\{ \phi ^{( \lambda , \nu )} \} $. Explicitly,
each element $ \ell \in \mathbb{F}^{*} $ sends $\phi $ to
$\ell ^{-1} \phi \, \ell $. Recall from the \text{Definition~\ref{Defn:isogeny}} that the conjugate action of each
$ \ell \in \mathbb{F}^{*} $ serves as an isomorphism in the sense of the
Drinfeld modules.

\begin{lem}%
\label{Lem:Dmiso}
The two Drinfeld modules $\phi ^{(\lambda ,\nu )}$ and
$\phi ^{(\lambda ', \nu ')}$ are isomorphism if and only if
%
\begin{equation}
\begin{cases}
&\lambda ^{q^{3}+q^{2}+q+1}= \lambda ^{\prime \,q^{3}+q^{2}+q+1}
\\
& \frac{\nu}{\lambda ^{q^{2}+1}}=\frac{\nu '}{\lambda ^{\prime \,q^{2}+1}}.
\end{cases}
\label{eq65}
\end{equation}
In particular, if $\nu =\nu '$ is fixed, then the isomorphism class of
$\phi $ depends on the invariant $ J = \lambda ^{q^{2}+1}$.
\end{lem}
\begin{proof}
From the \text{Definition~\ref{Defn:isogeny}}, $\phi ^{(\lambda ,\nu )}$ and
$\phi ^{(\lambda ', \nu ')}$ are isomorphic if and only if there exists
$\ell \in \mathbb{F}_{q^{4}}^{*}$, such that
\begin{equation*}
\phi ^{(\lambda ', \nu ')}=\ell ^{-1}\phi ^{(\lambda , \nu )}\ell ,
\end{equation*}
which is equivalent to
%
\begin{align}
\label{Eq:isphixy}
\begin{cases}
\phi _{x}^{(\lambda ', \nu ')}&=\ell ^{-1}\phi _{x}^{(\lambda , \nu )}
\ell
\\
\phi _{y}^{(\lambda ', \nu ')}&=\ell ^{-1}\phi _{y}^{(\lambda , \nu )}
\ell .
\end{cases}
\end{align}
Substituting Equation \text{\eqref{Eq:phixexp}} in the first formula of \text{\eqref{Eq:isphixy}} and comparing the coefficients of both sides, we obtain
the following conditions:
\begin{equation*}
\ell ^{q^{4}-1}=1;
\end{equation*}
%
\begin{align}
\label{co1}
&\frac{\lambda ^{\prime \,q^{4}}}{1-\zeta ^{1-q}}+
\frac{\lambda '}{1-\zeta ^{1-q}}=\ell ^{q^{3}-1}\Big(
\frac{\lambda ^{q^{4}}}{1-\zeta ^{1-q}}+
\frac{\lambda}{1-\zeta ^{1-q}}\Big);
\end{align}
%
\begin{align}
\label{Eq:co2}
&
\frac{\nu ' T^{q(q-1)(1-\sigma )-\sigma}\lambda ^{\prime \,q^{3}-q^{2}}}{\zeta ^{q}-\zeta}-
\frac{T \lambda ^{\prime \,1-q}}{(\zeta ^{q}-\zeta )\nu ' }+
\frac{T^{\sigma q-q}}{\zeta ^{q+1}\lambda ^{\prime \,q^{2}+q}}+
\frac{\zeta ^{q-1}\lambda ^{\prime \,q^{3}+1}}{(1-\zeta ^{q-1})^{2}}
\\
\nonumber
=&\ell ^{q^{2}-1}\Big(
\frac{\nu T^{q(q-1)(1-\sigma )-\sigma}\lambda ^{q^{3}-q^{2}}}{\zeta ^{q}-\zeta}-
\frac{T \lambda ^{1-q}}{(\zeta ^{q}-\zeta )\nu }+
\frac{T^{\sigma q-q}}{\zeta ^{q+1}\lambda ^{q^{2}+q}}+
\frac{\zeta ^{q-1}\lambda ^{q^{3}+1}}{(1-\zeta ^{q-1})^{2}}\Big);
\end{align}
and
%
\begin{align}
\label{Eq:co3}
\frac{T^{\sigma q}+T^{\sigma}}{\zeta \lambda ^{\prime \,q}}+
\frac{(T^{\sigma +q}+x)\lambda ^{\prime \,q^{2}-q+1}}{(1-\zeta ^{1-q})\nu '}=
\ell ^{q-1}\Big(\frac{T^{\sigma q}+T^{\sigma}}{\zeta \lambda ^{q}}+
\frac{(T^{\sigma +q}+x)\lambda ^{q^{2}-q+1}}{(1-\zeta ^{1-q})\nu}
\Big).
\end{align}
These conditions are equivalent to
%
\begin{align}
\ell ^{q-1}&=\frac{\lambda ^{q}}{\lambda ^{\prime \,q}};
\label{Eq:con1}
\\
\frac{\nu}{\lambda ^{q^{2}+1}}&=\frac{\nu '}{\lambda ^{\prime \,q^{2}+1}};
\label{Eq:con2}
\\
\lambda ^{q^{3}+q^{2}+q+1}&= \lambda ^{\prime \,q^{3}+q^{2}+q+1}.
\label{Eq:con3}
\end{align}
In the same fashion, combining Equation \text{\eqref{Eq:phiyexp}} with the second
formula of Equations \text{\eqref{Eq:isphixy}}, we have the same conditions \text{\eqref{Eq:con1}}-\text{\eqref{Eq:con3}}.
\end{proof}

The invariant $J $ in the \text{Lemma~\ref{Lem:Dmiso}} is called the $ J $-invariant
of $ \phi ^{(\lambda , \nu )} $. Set $ \ell = \lambda ^{q/(q-1)} $. Then
$ \ell $ induces an isomorphism from $ \phi ^{(\lambda , \nu )} $ to
$ \Phi ^{J} $, which is represented by
$ \Phi _{x}^{J},\, \Phi _{y}^{J} $ in Equations \text{\eqref{Eq:PhiJx}} and \text{\eqref{Eq:PhiJy}}. The family $\{\Phi ^{J}\}_{J}$ is completely determined
by the $ J $-invariant and includes all isomorphism classes of rank two
Drinfeld modules over $\bar K $, where the case of elliptic modular curves
over complex field is analogous. Combining with \text{Theorem~\ref{Thm:thphi}}, we complete the proof of \text{Theorem~\ref{Thm:cfPhiJ}}.

\subsection{Example}
\label{sec4.3}

In this example, we produce a special Drinfeld module over
$ \mathbb{F}_{q} = \mathbb{F}_{2} $ of rank two. Let
$\pi := t^{2} + t +1$ be the unique irreducible polynomial of degree two.
Let $ \zeta , ~\zeta +1 $ denote the roots of $ \pi $. The domain
$ \mathcal{A} $ is constructed by the relation
\begin{equation*}
y^{2} + xy + x^{2} + x = 0 .
\end{equation*}
Let $ \sqrt{t} $ be a square root of $ t $, and set
\begin{equation*}
\sqrt{x} := \frac{1}{t+1+\sqrt{t}}=
\frac{1}{(\sqrt{t}-\zeta )(\sqrt{t}-\zeta ^{2})}
\end{equation*}
and
\begin{equation*}
\sqrt{y} := \sqrt{t}\cdot \sqrt{x}.
\end{equation*}
It is clear that $ \sqrt{x}$ and $\sqrt{y}$, respectively, are square roots
of $x $ and $ y $. Let $ \psi $ be the Hayes module as in \text{Proposition~\ref{Prop:HM}} replacing $x$ and $y$ by $\sqrt{x}$ and $\sqrt{y}$ respectively.
That is
%
\begin{align}
\label{Eq:expsi}
\begin{cases}
\psi _{\sqrt{x}} = \tau ^{2} +\frac{\sqrt{x}+u^{3}}{u} \tau + \sqrt{x}
\\
\psi _{\sqrt{y}} = \zeta \tau ^{2} +\frac{\sqrt{y} +u^{3}\zeta }{u}
\tau + \sqrt{y},
\end{cases}
\end{align}
where $u$ is one of cubic roots of
$\frac{1}{(\sqrt{t}-\zeta )(t - \zeta )}$. Then the following pair forms
a $\zeta ^{q}$-type rank two Drinfeld module:
%
\begin{align}
\label{Eq:exphi}
\begin{cases}
\phi _{x} = \psi _{\sqrt{x}} \psi _{\sqrt{x}}
\\
\phi _{y} = \psi _{\sqrt{y}} \psi _{\sqrt{y}}.
\end{cases}
\end{align}
In the rest, we would like to show that $ \phi $ coincides with the Drinfeld
module $ \phi ^{(\lambda , \nu )} $ with
%
\begin{align}
\label{Eq:excmula}
\nu = \frac{ u }{ \zeta (\sqrt{t}- \zeta )}, ~\qquad ~~~~~~\qquad ~
\lambda = \frac{ \zeta }{ u (\sqrt{t}- \zeta ^{2})} .
\end{align}
One may directly show that $ \nu $ is a cubic root of
$ - T^{\sigma +2 }$. Substituting $ \psi _{\sqrt{y}}$ (see \text{Definition~\ref{Defn:psiu}}) into the expression of $ \phi _{y} $ in Equations \text{\eqref{Eq:exphi}}, we have
\begin{align*}
\phi _{y} & = (\zeta \tau - \frac{\sqrt{y}}{u})(\tau - u) (\zeta
\tau - \frac{\sqrt{y}}{u})(\tau - u)
\\
& = \zeta ^{2} ( \tau ^{2} + (\frac{\sqrt{y}}{u}+ \zeta ^{2} t u^{2} )
\tau + \frac{ \zeta y}{\sqrt{x}}) \psi _{\sqrt{x}}.
\end{align*}
We discover that
%
\begin{align}
\label{Eq:exphiIinf}
\phi _{I_{\infty}}= \tilde{\phi}_{I_{\infty}} = \psi _{\sqrt{x}}=
\tau ^{2} +\frac{\sqrt{x}+u^{3}}{u} \tau + \sqrt{x}
\end{align}
and
%
\begin{align}
\label{Eq:extphiI0}
\tilde{\phi}_{I_{0}} = \tau ^{2} + (\frac{\sqrt{y}}{u}+ \zeta ^{2} t u^{2}
) \tau + \frac{ \zeta y}{\sqrt{x}} .
\end{align}
Using the same notations as in Equations \text{\eqref{Eq:tphiiinf}} and \text{\eqref{Eq:tphii0}}, we let
%
\begin{align}
\tilde{\alpha} :=& \frac{\sqrt{x}+u^{3}}{u} ;
\label{eq78}
\\
\tilde{\beta} := & \frac{1}{u} ( \sqrt{y} +\zeta ^{2} t u^{3} );
\label{eq79}
\\
\delta := & \sqrt{x} =
\frac{1}{(\sqrt{t}-\zeta )(\sqrt{t}-\zeta ^{2})}.
\label{eq80}
\end{align}
A simple calculation yields
\begin{equation*}
\lambda := \tilde{\beta} -\tilde{\alpha} =
\frac{ \zeta }{ u (\sqrt{t}- \zeta ^{2})} .
\end{equation*}
Then it follows from $ \delta = \nu \lambda $ that
\begin{equation*}
\nu = \frac{ u }{ \zeta (\sqrt{t}- \zeta )}.
\end{equation*}
So the $ J $-invariant of $ \phi $ is given by
\begin{equation*}
J (\phi ) = \lambda ^{5} =
\frac{\nu (\sqrt{t} - \zeta )^{3}}{\sqrt{t} - \zeta ^{2}}.
\end{equation*}
%

\section{Weil pairing}
\label{Sec:WP}

In this section, we would like to represent the Weil pairing of rank
$r$ Drinfeld modules by a so-called Drinfeld-Moore product. The principle
technique is to derive the Weil operator associated to the ideal $I$ of
$\mathcal{A}$, which is essentially constructed by the basis and the dual
basis of $\mathcal{A}/I$. An explicit formula for Weil pairing of the rank
two Drinfeld modules is demonstrated by employing the concrete expression
of the rank two Weil operators over $\mathcal{A}$.

\subsection{Dual basis}
\label{Sec:DB}

Now we compute the basis and the dual basis of $ \mathcal{A}/I $. Following
the approach in \cite{vdHGJ04}, we represent the dual basis by means of
torsion differential module.

Let $ \Omega _{\mathcal{A}} $ be the differential module of
$ \mathcal{A} $. Since the differential of $ \rho (x,y) $ in Equation \text{\eqref{Eq:rho}} is given by
\begin{equation*}
2 y dy - \mathrm{Tr}(\zeta ) (ydx+xdy ) + (2 \zeta ^{q+1} x - 1 ) dx ,
\end{equation*}
it follows that $ \Omega _{\mathcal{A}} $ is isomorphic to
\begin{equation*}
\mathcal A d x \oplus \mathcal A dy / \big( 2 y dy - \mathrm{Tr}(
\zeta ) (ydx+xdy ) + (2 \zeta ^{q+1} x - 1 ) dx \big),
\end{equation*}
as an $\mathcal{A}$-module. Set $\omega _{\pi}:=\frac{dt }{\pi} $. One
may check that $\omega _{\pi}$ is the generator of
$ \Omega _{\mathcal{A}}$, namely,
\begin{equation*}
\Omega _{\mathcal{A}} \cong \mathcal{A} \cdot \omega _{\pi}.
\end{equation*}
Now we fix the ideal $ I $ of $ \mathcal{A}$. Let
$ I^{-1}\Omega _{\mathcal{A}}$ be a fractional module in
$K \otimes _{\mathcal{A}} \Omega _{\mathcal{A}} $ and
$ \Omega _{I} := I^{-1} \Omega _{\mathcal{A}} / \Omega _{\mathcal{A}} $
the correspondent quotient module.

\begin{notation}
Denote by $ \operatorname{Res}_{\zeta}(\eta ) $ the residue number of
$ \eta $ at the zero of $ t-\zeta $ for a differential
$ \eta \in K \otimes _{\mathcal{A}} \Omega _{\mathcal{A}} = K\cdot dt $.
By abuse of language, we make use of the notation
$ \operatorname{Res}_{\zeta}(\omega ) := \operatorname{Res}_{\zeta}(1
\otimes \omega ) $ for $ \omega \in \Omega _{\mathcal{A}}$.
\end{notation}

\begin{lem}%
\label{Lem:pp}
There exists a perfect pairing
\begin{equation*}
\langle \cdot ,\cdot \rangle : ~ \mathcal{A}/ I \otimes _{\mathbb F_{q}}
\Omega _{I} \to \mathbb{F}_{q}.
\end{equation*}
Therefore, the dual space of $\mathcal{A}/I$ is isomorphic to
$\Omega _{I}$.
\end{lem}

\begin{proof}
Indeed this result is well-known, so we only sketch a proof. It is easy
to verify that the residue map
\begin{equation*}
\operatorname{Res}_{\zeta}: ~ \mathcal{A}\otimes _{\mathbb F_{q}} I^{-1}
\Omega _{I} \to \mathbb{F}_{q}(\zeta )
\end{equation*}
induces a non-degenerate pairing
\begin{equation*}
\operatorname{Res}_{\zeta}: ~ \mathcal{A}/I \otimes _{\mathbb F_{q}}
\Omega _{I} \to \mathbb{F}_{q}(\zeta ).
\end{equation*}
The composition with the trace map
$\mathrm{Tr}: {\mathbb F_{q}}(\zeta ) \to {\mathbb F_{q}} $ gives a perfect
pairing.
\end{proof}

Additionally, we assume that $ I $ is principle and generated by the element
%
\begin{equation}
\label{Eq:P}
P(x,y) = \alpha _{d}+ \sum _{j=0}^{d-1} \Big(\alpha _{j} x^{d-j} +
\beta _{j} y x^{d-j-1}\Big),
\end{equation}
where the coefficients
$ \alpha _{j} ,~ \beta _{j}, ~\alpha _{d} \in \mathbb{F}_{q} $ for
$ j = 0 , \ldots , d-1 $. Then the differential
%
\begin{equation}
\label{Eq:OmstarP}
\omega _{*} := \frac{1}{P(t)} \omega _{\pi }%
\end{equation}
serves as a generator of $ I^{-1} \Omega _{\mathcal{A}}$. It is natural
to choose the bases $\mathcal{B}_{1} $ and $ \mathcal{B}_{2}$ of
$ \mathcal{A}/ I $, respectively, in the following two cases. For
$ \alpha _{0} \not = 0 $, we choose
\begin{equation*}
\mathcal{B}_{1} = \left \{ 1,\, y,\, x,\, yx,\, x^{2} ,\, \ldots , \, y
x^{d-2},\, x^{d-1},\, y x^{d-1} \right \} ;
\end{equation*}
and for $ \beta _{0} \not = 0 $, set
\begin{equation*}
\mathcal{B}_{2} = \left \{ 1,\, y,\, x, \, yx ,\, \ldots , \, yx^{d-2},
\, x^{d-1},\, x^{d} \right \} .
\end{equation*}
Notice that the only differences of two bases are exactly the last terms.

\begin{notation}%
\label{No:dbwv}
\begin{enumerate}
\item For determining the dual basis of $ \mathcal B_{1} $, we introduce
the polynomials $ w(e)$ in $ \mathbb{F}_{q}[x,y] $ generated by
$\mathcal B_{1}$, where the label $ e $ ranges over
$ \mathcal{B}_{1} $ and splits into the following four types.
\begin{enumerate}[Type 1:]
\item[Type 1:] $e=x^{d-1}y$,
%
\begin{align}
\label{Eq:wxy}
w(x^{d-1}y ) =
\frac {(\beta _{0}\zeta + \alpha _{0} )(\beta _{0} \zeta ^{q} + \alpha _{0} )}{\alpha _{0}};
\end{align}
\item[Type 2:] $e=x^{d-1-j} y$ for $ j=1,\ldots , d-1$,
%
\begin{align}
\label{Eq:wxjy}
w(x^{d-1-j} y) = \sum _{k=0}^{j-1} (\beta _{k} x^{j-k-1} y + \alpha _{k}
x^{j-k})+ \Big ( \alpha _{j}+ (\mathrm{Tr}\zeta +
\frac{\beta _{0}}{\alpha _{0}}\zeta ^{q+1}) \beta _{j} \Big);
\end{align}
\item[Type 3:] $e=x^{d-j}$ for $ j=1,\ldots , d-1$,
%
\begin{align}
\label{Eq:wxj}
\nonumber
w(x^{d-j}) =& \sum _{k=0}^{j-1} ( (x^{j-1-k}y-\mathrm{Tr}\zeta x^{j-k})
\alpha _{k} + x^{j-k}\beta _{k-1} -x^{j-k} \zeta ^{q+1}\beta _{k} )
\\
& + \Big ( \beta _{j-1}- \zeta ^{q+1}\beta _{j} +
\frac{\alpha _{j}}{\alpha _{0}} \zeta ^{q+1} \beta _{0} \Big);
\end{align}
\item[Type 4:] $e=1$,
%
\begin{align}
\label{Eq:w1}
\begin{aligned}
w(1) = & \sum _{k=1}^{d-1} \biggl ( \big( (\mathrm{Tr}\zeta +
\frac{\beta _{0} \zeta ^{q+1}}{\alpha _{0} } ) \beta _{k} + \alpha _{k}
\big) x^{d-k-1}y + (\frac{\beta _{0} \zeta ^{q+1}}{\alpha _{0} }
\alpha _{k} + \beta _{k-1}\\
& - \zeta ^{q+1} \beta _{k}) x^{d-k} \biggr )
 +\left ( \beta _{d-1} + \mathrm{Tr}(\zeta ) \alpha _{d} + 2 \zeta ^{q+1}
\frac{\alpha _{d}}{\alpha _{0}} \beta _{0} \right )\\
& + ( (\mathrm{Tr}
\zeta + \frac{\beta _{0} \zeta ^{q+1}}{\alpha _{0} } ) \beta _{0} +
\alpha _{0}) x^{d-1}y.
\end{aligned}
\end{align}
\end{enumerate}
\item Similarly, in order to calculate the dual basis of
$ \mathcal B_{2}$, we list $v(e)\in \mathbb{F}_{q}[x,y] $ for $ e $ belonging
to the following four types.
\begin{enumerate}[Type 1:]
\item[Type 1:] $e=x^{d}$,
%
\begin{align}
\label{Eq:vxd}
v(x^{d} ) = -
\frac{(\beta _{0}\zeta + \alpha _{0} )(\beta _{0} \zeta ^{q} + \alpha _{0} )}{\beta _{0} };
\end{align}
\item[Type 2:] $e=x^{d-1-j}y$ for $j=1,\ldots , d-1 $,
%
\begin{align}
\label{Eq:vxjy}
v(x^{d-1-j} y ) = \sum _{k=0}^{j-1} (\beta _{k} yx^{j-k-1} + \alpha _{k}
x^{j-k} ) + \Big ( \alpha _{j}-
\frac{\alpha _{0} \beta _{j} }{\beta _{0}} \Big);
\end{align}
\item[Type 3:] $e=x^{d-j}$ for $j=1,\ldots , d-1$,
%
\begin{align}
\label{Eq:vxj}
\nonumber
v(x^{d-j})&= \sum _{k=0}^{j-1} ( (yx^{j-k-1}-\mathrm{Tr}(\zeta ) x^{j-k})
\alpha _{k} + \beta _{k-1}x^{j-k} - \zeta ^{q+1}\beta _{k} x^{j-k} )
\\
& + \Big ( \beta _{j-1}- \zeta ^{q+1}\beta _{j} - \big(\mathrm{Tr}
\zeta +\frac{\alpha _{0}}{\beta _{0}}\big) \alpha _{j} \Big);
\end{align}
\item[Type 4:] $e=1$,
%
\begin{align}
\label{Eq:v1}
\nonumber
v(1) =& \sum _{k=1}^{d-1} \left ( (\alpha _{k} -
\frac{\alpha _{0}}{\beta _{0} } \beta _{k} ) x^{d-k-1}y -\Big((
\mathrm{Tr}\zeta +\frac{\alpha _{0}}{\beta _{0} }) \alpha _{k} -
\beta _{k-1} + \zeta ^{q+1} \beta _{k}\Big) x^{d-k}\! \right )
\\
& +\Big( - \mathrm{Tr}\zeta \alpha _{d} + \beta _{d-1} - 2
\frac{\alpha _{0}\alpha _{d}}{ \beta _{0}} \Big)-\Big((\mathrm{Tr}
\zeta +\frac{\alpha _{0}}{\beta _{0} }) \alpha _{0} + \zeta ^{q+1}
\beta _{0}\Big) x^{d}.
\end{align}
\end{enumerate}
\end{enumerate}
\end{notation}

The facts concerning residue numbers below would be very useful.

\begin{remark}%
\label{Rem:res}
Let $G(t)\in \mathbb F_{q}[t]$ be a polynomial. If
$ \deg (G(t)) \leqslant 2 j $, then the differential
$\frac{G(t)}{\pi ^{j}}\omega _{\pi}$ is regular except for the place
$ P_{\pi }$. It follows from the residue theorem that
%
\begin{align}
\label{Eq:resGom1}
\mathrm{Tr}\operatorname{Res}_{\zeta}\frac {G(t)} {\pi ^{j}} \cdot
\omega _{\pi }= 0 \quad \text{when $\deg (G(t)) \leqslant 2 j$}.
\end{align}
The differential form $G(t)\pi ^{k}\omega _{\pi}$ (with
$k\geqslant 1$) is regular at $P_{\pi}$, so
%
\begin{align}
\label{Eq:resGom2}
\mathrm{Tr}\operatorname{Res}_{\zeta}G(t)\pi ^{k} \cdot \omega _{\pi }=
0 \quad \text{when $k\geqslant 1$}.
\end{align}
Furthermore, it is fairly easy to verify that
%
\begin{align}
\label{Eq:res3}
\mathrm{Tr}\operatorname{Res}_{\zeta} (a+bt+c t^{2})\omega _{\pi }= b +
c\mathrm{Tr}\zeta .
\end{align}
\end{remark}

Now we compute the dual bases for the both cases by applying the perfect
pairing in \text{Lemma~\ref{Lem:pp}}.

\begin{prop}%
\label{Prop:Thbd}
As above, we let $\mathcal B_{1} $ be the basis of $\mathcal{A}/I$ when
$\alpha _{0}\neq 0$ together with polynomials $w(e)$ for
$e\in \mathcal{B}_{1}$. Then the dual basis of $\mathcal{B}_{1}$ is represented
by the set
\begin{equation*}
\left \{ w(e)\cdot \omega _{*} | e \in \mathcal B_{1} \right \}.
\end{equation*}
Furthermore, for $ \beta _{0} \not = 0 $, there exists an analogous dual
basis for the second basis $ \mathcal B_{2} $ in terms of polynomials
$ v(e) $ with $e\in \mathcal{B}_{2}$.
\end{prop}
\begin{proof}
Since the second assertion can be shown in the same way, we only present
the proof for the first assertion. For each
$e,\, f \in \mathcal{B}_{1}$, denote the Kronecker delta symbol by
\begin{equation*}
\delta (e,f) =
\begin{cases}
0 , & e \not = f
\\
1, & e = f.
\end{cases}
\end{equation*}
The proposition asserts the equality
%
\begin{equation}
\label{Eq:delef}
\delta (e,f) = \left \langle{ f, w(e) \omega _{*}} \right \rangle
\quad \text{for $e,\, f\in \mathcal{B}_{1}$} .
\end{equation}
We split the procedure for verifying the equality \text{\eqref{Eq:delef}} into
four cases depending on the types listed in \text{Notation~\ref{No:dbwv}}.

Case 1: $e = x^{d-1} y $.

First we let $F(t)=P(x,y)\pi ^{d} $, i.e.
%
\begin{align}
\label{Eq:Ft}
F(t) = (\beta _{0} t + \alpha _{0}) + (\beta _{1} t + \alpha _{1})
\pi + \cdots + (\beta _{d-1} t + \alpha _{d-1} ) \pi ^{d-1} +\alpha _{d}
\pi ^{d}.
\end{align}
Clearly, $F(t)$ is of degree $ \leqslant 2 d $ and coprime with
$ \pi (t) $. The generator $\omega _{*}$ of $\Omega _{I}$ in Equation \text{\eqref{Eq:OmstarP}} is now rewritten as
%
\begin{align}
\label{Eq:omstarF}
\omega _{*} =\frac{\pi ^{d}}{F(t)} \omega _{\pi}.
\end{align}
From Equation \text{\eqref{Eq:resGom2}} in \text{Remark~\ref{Rem:res}}, we obtain
%
\begin{align}
\label{Eq:Thresxj}
\left \langle{x^{d-1-j}y, \omega _{*} } \right \rangle = \mathrm{Tr}
\operatorname{Res}_{\zeta} \frac{t \pi ^{j}}{F(t)} \omega _{\pi }= 0
\quad \text{ for $ 1\leqslant j \leqslant d-1$}.
\end{align}
Applying Equation \text{\eqref{Eq:resGom2}} again yields
%
\begin{align}
\label{Eq:Thresxk}
\left \langle{x^{d-k}, \omega _{*} } \right \rangle = \mathrm{Tr}
\operatorname{Res}_{\zeta} \frac{\pi ^{k}}{F(t)}\omega _{\pi }= 0
\quad \text{for $ 1 \leqslant k \leqslant d$}.
\end{align}
It follows from Equation \text{\eqref{Eq:res3}} that
%
\begin{align}
\label{Eq:Thresxy}
\left \langle{x^{d-1}y ,
\frac{(\beta _{0}\zeta + \alpha _{0} )(\beta _{0} \zeta ^{q} + \alpha _{0} )}{\alpha _{0} }
\omega _{*} } \right \rangle =
\frac{(\beta _{0}\zeta + \alpha _{0} )(\beta _{0} \zeta ^{q} + \alpha _{0} )}{\alpha _{0} }
\mathrm{Tr}\operatorname{Res}_{\zeta} \frac{t}{F(t)} \omega _{\pi }= 1.
\end{align}
Therefore, Equations \text{\eqref{Eq:Thresxj}}-\text{\eqref{Eq:Thresxy}} confirm the
equality \text{\eqref{Eq:delef}} for the case $e=x^{d-1}y$.

Case 2: $e= x^{d-1-j}y$, for $j=1\,, \ldots ,\, d-1$.

We would like to produce a series of polynomials
$ \{ G_{j}(t)\}_{1 \leqslant j \leqslant d-1}$ of degree
$\deg (G_{j}(t)) \leqslant 2j$ satisfying the following $\pi $-expansion
condition
%
\begin{align}
\label{Eq:Gjt1}
\frac{G_{j}(t)}{F(t)} = 1 + (t-\mathrm{Tr}{\zeta})g_{j} \pi ^{j} + O(
\pi ^{j+1}),
\end{align}
where $O(\pi ^{j+1})$ denotes the higher order term and
$g_{j}\in \mathbb F_{q}$ needs to be determined.

Substituting expression of $F(t)$ (see Equation \text{\eqref{Eq:Ft}}) in
$tF(t)$ yields
%
\begin{align}
\label{Eq:tFt}
\nonumber
tF(t) =& \Big ( (\beta _{0} \mathrm{Tr}\zeta + \alpha _{0}) t -
\zeta ^{q+1} \beta _{0} \Big) + \Big ( (\beta _{1} \mathrm{Tr}\zeta +
\alpha _{1}) t + \beta _{0} - \zeta ^{q+1} \beta _{1} \Big) \pi +
\cdots
\\
& + \Big ( (\beta _{d-1} \mathrm{Tr}\zeta + \alpha _{d-1}) t + \beta _{d-2}-
\zeta ^{q+1}\beta _{d-1} \Big) \pi ^{d-1} +(\beta _{d-1} + \alpha _{d}
t) \pi ^{d}.
\end{align}
It follows from Equations \text{\eqref{Eq:Gjt1}} and \text{\eqref{Eq:tFt}} that
%
\begin{align}
G_{j}(t)=& F(t) + \big(tF(t)-\mathrm{Tr}(\zeta )F(t)\big)g_{j} \pi ^{j}
+O(\pi ^{j+1})
\nonumber
\\
= & \sum _{k=0}^{j-1} (\beta _{k} t + \alpha _{k})\pi ^{k} + \Big ( (
\beta _{j}+ \alpha _{0} g_{j}) t + \alpha _{j} - g_{j}\zeta ^{q+1}
\beta _{0} -g_{j} \mathrm{Tr}\zeta \alpha _{0} \Big) \pi ^{j}.
\label{Eq:Gjt2}
\end{align}
With the assumption $\deg (G_{j}(t)) \leqslant 2j$, the coefficient of
$t\pi ^{j}$ in $G_{j}(t)$ should be $0$, i.e.
$ g_ {j}= -\frac{\beta _{j}}{\alpha _{0}}$. Thus, Equations \text{\eqref{Eq:omstarF}} and \text{\eqref{Eq:Gjt1}} show
\begin{equation*}
\frac{G_{j}(t)}{\pi ^{j}} \cdot \omega _{*}= \frac{G_{j}(t)}{F(t)}
\cdot \pi ^{d-j}\omega _{\pi}=\Big( 1 -\frac{\beta _{j}}{\alpha _{0}} (t-
\mathrm{Tr}{\zeta}) \pi ^{j} \Big) \pi ^{d-j} \omega _{\pi }+O(\pi ^{d+1})
\omega _{\pi}.
\end{equation*}
By combining Equations \text{\eqref{Eq:resGom1}}-\text{\eqref{Eq:res3}} it follows that
\begin{align*}
\left \langle{ f , \frac{G_{j}(t)}{\pi ^{j}} \cdot \omega _{*} }
\right \rangle = \delta (f, x^{d-1-j} y ), \quad \text{for each} ~f
\in \mathcal B_{1}.
\end{align*}
Using Equation \text{\eqref{Eq:Gjt2}} we achieve the equality
%
\begin{align}
\label{Eq:ThGw}
\frac{G_{j}(t)}{\pi ^{j}} = w(x^{d-1-j}y ).
\end{align}

Case 3: $e= x^{d-j}$, for each $j=1, \ldots , d-1$.

In this case, we consider the polynomials
$\{ H_{j}(t)\}_{1 \leqslant j \leqslant d-1}$ such that
%
\begin{align}
\label{Eq:Hjt1}
\frac{H_{j}(t)}{F(t)} = (t-\mathrm{Tr}{\zeta}) + (t-\mathrm{Tr}{\zeta})h_{j}
\pi ^{j} +O( \pi ^{j+1}),
\end{align}
where $\deg (H_{j}(t)) \leqslant 2j$ and $h_{j}\in \mathbb F_{q}$. Applying
the same arguments as in Case 2, we have
$ h_{j} = -\frac{\alpha _{j}}{\alpha _{0}}$ and
%
\begin{multline}
\label{Eq:ThHt}
H_{j}(t) = \sum _{k=0}^{j-1} ( (t-\mathrm{Tr}\zeta ) \alpha _{k} +
\beta _{k-1} - \zeta ^{q+1}\beta _{k} ) \pi ^{k} + \Big ( \beta _{j-1}-
\zeta ^{q+1}\beta _{j} + \frac{\alpha _{j}}{\alpha _{0}} \zeta ^{q+1}
\beta _{0} \Big) \pi ^{j} .
\end{multline}
Moreover,
\begin{align*}
\frac{H_{j}(t)}{\pi ^{j}} = w(x^{d-j})
\end{align*}
satisfies the equality
\begin{align*}
\left \langle{ f, \frac{H_{j}(t)}{\pi ^{j}} \cdot \omega _{*} }
\right \rangle =\delta (f, x^{d-j}) , \qquad ~ \text{for each }f \in
\mathcal{B}_{1} .
\end{align*}
Case 4: $e= 1$.

We let $I_{d}(t)$ be polynomial of degree $ \leqslant 2d$ such that
%
\begin{align}
\label{Eq:Idt1}
\frac{I_{d}(t)}{F(t)} &= t + g_{d} + (t -\mathrm{Tr}{\zeta}) h_{d}
\pi ^{d} +O(\pi ^{d+1}),
\end{align}
and
%
\begin{align}
\label{Eq:Idt2}
I_{d}(t) = ((\mathrm{Tr}\zeta +
\frac{\beta _{0} \zeta ^{q+1}}{\alpha _{0} } ) \beta _{0} + \alpha _{0})t
+O(\pi ).
\end{align}
As above, one can verify that
$ g_{d} = \frac{\beta _{0} \zeta ^{q+1}}{\alpha _{0} } $ and
$ h_{d} = -\frac{\alpha _{d}}{\alpha _{0}}$ under the assumption that
$\frac{I_{d}(t)}{\pi ^{d}}$ is generating by the basis
$\mathcal{B}_{1}$. In addition, a direct calculation shows that
%
\begin{align}
\label{Eq:ThIt}
\nonumber
I_{d}(t) =& \sum _{k=0}^{d-1} \left ( ( (\mathrm{Tr}\zeta +
\frac{\beta _{0} \zeta ^{q+1}}{\alpha _{0} } ) \beta _{k} + \alpha _{k})
t + (\frac{\beta _{0} \zeta ^{q+1}}{\alpha _{0} } \alpha _{k} +
\beta _{k-1} - \zeta ^{q+1} \beta _{k}) \right ) \pi ^{k}
\\
& +\left ( \beta _{d-1} + \mathrm{Tr}\zeta \alpha _{d} + 2 \zeta ^{q+1}
\frac{\alpha _{d}}{\alpha _{0}} \beta _{0} \right ) \pi ^{d},
\end{align}
which implies that
%
\begin{align}
\label{Eq:ThIdw}
\frac{I_{d}(t)}{\pi ^{d}} = w(1)
\end{align}
and
%
\begin{align}
\label{Eq:ThIw}
\left \langle{ f, \frac{I_{d}(t)}{\pi ^{d}}\cdot \omega _{*}} \right
\rangle =\delta (f, 1), \quad \quad \text{for each}~ f \in \mathcal{B}_{1}.
\end{align}
This completes the proof.
\end{proof}

\subsection{Representation of Weil pairing}
\label{sec5.2}

Our goal in this section is devoted to presenting the Weil pairing by means
of the motive structure of Drinfeld modules. Firstly, we need to introduce
the Motive representation of torsion points; for extended details one should
refer to \cite{vdHGJ04}. We maintain some notations from the previous sections
below. Let $\phi $ be a Drinfeld module of rank $r$, $I$ an ideal of
$\mathcal{A}$, and $\phi _{I} $ an annihilator of $ I $. Recalling
$\bar{M}(\phi )$ the motive structure of $\phi $, there exists an isomorphism
\begin{equation*}
\ker (\phi _{I}) \cong \mathrm{Hom}_{\mathcal{A}/I }((\bar{M}(\phi )/I
\bar{M}(\phi ))^{\tau }, \Omega _{I} ).
\end{equation*}
Through Lang's Theorem in \cite{LS56}, i.e., the
$ \mathcal{A}\otimes _{\mathbb F_{q}} \bar{K}\{\tau \} $-isomorphism
\begin{equation*}
(\bar{M}(\phi )/ I \bar{M}(\phi ))^{\tau }\otimes _{\mathbb{F}_{q}}
\bar{K}\cong \bar{M}(\phi ) / I \bar{M}(\phi ),
\end{equation*}
it induces a natural isomorphism
%
\begin{equation}
\label{Equ:isomorphism}
\ker (\phi _{I}) \cong \mathrm{Hom}_{\mathcal{A}\otimes \bar{K}\{
\tau \}}(\bar{M}(\phi ), \Omega _{I}\otimes \bar{K}).
\end{equation}
%
\begin{defn}
\label{defn5.6}
Given a torsion $ \mu \in \ker (\phi _{I}) $, the equivalent homomorphism
\begin{equation*}
h_{\mu }\in \mathrm{Hom}_{\mathcal{A}\otimes \bar{K}\{\tau \}}(
\bar{M}(\phi ), \Omega _{I}\otimes \bar{K}),
\end{equation*}
via the isomorphism \text{\eqref{Equ:isomorphism}} is called the motive representation
of $ \mu $.
\end{defn}
Suppose that $ \{x_{i}\}_{i=1}^{n} $ is a basis of $ \mathcal{A}/ I$ along
with the dual basis $ \{\omega _{j}\}_{j=1}^{n} \in \Omega _{I} $. Then
$h_{\mu }$ can be written as
%
\begin{align}
\label{Eq:hmr}
h_{\mu }(\tau ^{k}) = \sum _{j=1}^{n} \omega _{j} \otimes \tau ^{k}(
\phi _{x_{j}}(\mu )),
\end{align}
where $ \tau ^{k} \in \bar {M}(\phi ) $ with $k\in \mathbb{N}$.

Secondly, we define the notion of Drinfeld-Moore product based on the Moore
determinant.
%
\begin{defn}
\label{defn5.7}
The Moore determinant of variables $ x_{1},\ldots , x_{r} $ means the following
determinant
\begin{equation*}
\mathcal{M}(x_{1},\ldots ,x_{r}) = \det
\begin{pmatrix}
x_{1} & x_{1}^{q} & \cdots & x_{1}^{q^{r-1}}
\\
x_{2} & x_{2}^{q} & \cdots & x_{2}^{q^{r-1}}
\\
\vdots & \vdots & \cdots & \vdots
\\
x_{r} & x_{r}^{q} & \cdots & x_{r}^{q^{r-1}}
\end{pmatrix}
.
\end{equation*}
\end{defn}
%
\begin{defn}[Drinfeld-Moore Product]
\label{defn5.8}
Let $\phi $ be a Drinfeld module of rank $ r $ over the
$\mathcal{A}$-field $ L $. Given
$ \sum _{i} a_{1}^{(i)} \otimes a_{2}^{(i)} \otimes \cdots \otimes a_{r}^{(i)}
\in {\otimes}^{r}_{\mathbb{F}_{q}} \mathcal{A}$ and
$ \mu _{1},\mu _{2}, \ldots , \mu _ {r} \in \bar{L} $, the Drinfeld-Morre
product
\begin{equation*}
{\otimes}^{r}_{\mathbb F_{q}} \mathcal{A}\times \oplus ^{r}_{
\mathbb F_{q}}(\bar{L} ) \to \bar{L}
\end{equation*}
is represented by
\begin{equation*}
(\sum _{i} a_{1}^{(i)} \otimes a_{2}^{(i)} \otimes \cdots \otimes a_{r}^{(i)}
) *_{\phi }(\mu _{1},\mu _{2}, \ldots , \mu _{r} ) := \sum _{i}
\mathcal{M}(\phi _{a_{1}^{(i)}} (\mu _{1}) ,\phi _{a_{2}^{(i)}}(\mu _{2}),
\ldots , \phi _{a_{r}^{(i)}} (\mu _{r}) ),
\end{equation*}
where $ \mathcal{M} $ means the Moore determinant.
\end{defn}
It is trivial to check that the Drinfeld-Moore product satisfies the
$\mathbb{F}_{q} $-multilinear property. That is, for
$\mathbf{a} , \mathbf{b} \in {\otimes}^{r}_{\mathbb{F}_{q}}
\mathcal{A}$,
$ \mu _{1}, \mu _{1}',\ldots , \mu _{r} ,\mu _{r}' \in \bar{L} $,
\begin{equation*}
(\mathbf{a}+\mathbf{b})*_{\phi} (\mu _{1}, \ldots , \mu _{r})=
\mathbf{a} *_{\phi}(\mu _{1}, \ldots , \mu _{r}) +\mathbf{b} *_{\phi} (
\mu _{1}, \ldots , \mu _{r})
\end{equation*}
and
\begin{equation*}
\mathbf{a} *_{\phi} (\mu _{1}, \ldots ,(\mu _{i}+\mu _{i}') ,\ldots ,
\mu _{r})= \mathbf{a} *_{\phi}(\mu _{1}, \ldots , \mu _{i}, \ldots ,
\mu _{r}) +\mathbf{a} *_{\phi} (\mu _{1},\ldots ,\mu _{i}',\ldots ,
\mu _{r}) .
\end{equation*}
Using this notation, the Moore determinant can be written as
$ \mathbf{1} *_{\phi} (\mu _{1},\ldots , \mu _{r})$, where
$\mathbf{1}$ is the unit of
$\otimes ^{r}_{\mathbb F_{q}} \mathcal{A}$. For $r = 1$, the Drinfeld-Moore
product reduces to the standard action of $ \mathcal{A}$ on
$ \bar{L} $, see Equation \text{\eqref{Eq:actiona}}. We can produce the Weil pairing
via the Drinfeld-Moore product.
%
\begin{thm}%
\label{Thm:wo}
Given an ideal $ I $ of $\mathcal{A}$, there exists a symmetric form
\begin{equation*}
\operatorname{WO}_{I}^{(r)} \in \operatorname{Sym}^{r}_{\mathbb F_{q}}
\mathcal{A}\subseteq {\otimes}_{\mathbb{F}_{q}}^{r} \mathcal{A},
\end{equation*}
depending on $ I $ such that the Weil pairing of any rank $r$ Drinfeld
module $ \phi $ over $\mathcal{A}$-field $L$ is represented by
\begin{align*}
\wedge ^{r}\ker (\phi _{I}) &\to \bar{L}
\\
(\mu _{1},\ldots , \mu _{r} ) &\mapsto \operatorname{WO}_{I}^{(r)} *_{
\phi}(\mu _{1},\ldots , \mu _{r} ),
\end{align*}
where $ \mu _{i}$'s denotes the torsion points of $\phi _{I} $. In particular,
if $ \psi = \wedge ^{r} \phi $ denotes the symmetric product, then
$ \operatorname{WO}_{I}^{(r)} *_{\phi}(\mu _{1},\ldots , \mu _{r} ) $ is
annihilated by $ \psi _{I} $.
\end{thm}
\begin{proof}
Suppose that $\phi _{I} $ is an annihilator of $ I $. Let us fix some torsion
points $ \mu _{1}, \ldots , \mu _{r} $ and a basis
$ \{ e_{j} \}_{j=1}^{n} $ of $ \mathcal{A}/ I $. Denote by
$ \{\omega _{i} \}_{i=1}^{n} $ the dual basis of
$ \{ e_{j} \}_{j=1}^{n} $, namely,
\begin{equation*}
\langle e_{i} , \omega _{j} \rangle = \delta (i,j) .
\end{equation*}
As discussed previously, we let $h_{\mu _{i}}$ be the motive representation
of $ \mu _{i} $, i.e.
\begin{equation*}
h_{i}( \tau ^{k} ) = \sum _{j=1}^{n} \phi _{e_{j}} (\mu _{i} )^{q^{k}}
\omega _{j}.
\end{equation*}
Denote by $ S^{r} $ the symmetric group on numbers
$ 0,1,\ldots , r-1 $ and $\varepsilon (s)$ the permutation symbol for
$s \in S^{r}$. Assume that
$ 1_{\psi }:= \tau ^{0} \wedge \tau ^{1} \wedge \cdots \wedge \tau ^{r-1}$
is the basis of $ \bar{M}(\psi ) $. According to the definition of Drinfeld-Moore
product, for $s=(s_{0}, s_{2}, \ldots , s_{r-1})\in S^{r}$ we have
%
\begin{align}
(h_{1} \wedge h_{2} \wedge \cdots \wedge h_{r} ) (1_{\psi }) & =
\sum _{s \in S^{r}} \varepsilon (s) h_{1}(\tau ^{s_{0}}) \otimes h_{2}(
\tau ^{s_{2}}) \otimes \cdots \otimes h_{r}(\tau ^{s_{r-1}})
\nonumber
\\
& = \sum _{1\leqslant j_{1},\ldots , j_{r}\leqslant n} \mathcal{M} (
\phi _{e_{j_{1}}}(\mu _{1}), \ldots , \phi _{e_{j_{r}} }(\mu _{r}))
\otimes \omega _{j_{1}} \otimes \cdots \otimes \omega _{j_{r}}
\nonumber
\\
& = \sum _{ 1\leqslant j_{1},\ldots , j_{r-1}\leqslant n} H_{j_{1},j_{2},
\ldots , j_{r-1}} \otimes \omega _{j_{1}} \otimes \cdots \otimes
\omega _{j_{r-1}},
\label{Eq:h1hr}
\end{align}
where
%
\begin{align}
\label{ThH1}
H_{j_{1},\ldots , j_{r-1}} = \sum _{k=1} ^{n}\mathcal{M} \left (\phi _{e_{j_{1}}}(
\mu _{1}), \ldots ,\phi _{e_{j_{r-1}} }(\mu _{r-1}), \phi _{e_{k} }(
\mu _{r}) \right ) \otimes \omega _{k}.
\end{align}
From Equation \text{\eqref{Equ:isomorphism}}, there exists a unique torsion
$ \nu _{j_{1},\ldots , j_{r-1}} \in \ker \psi _{I} $ corresponding to
$H_{j_{1},\ldots , j_{r-1}}$ for each subscript
$ j_{1},\ldots , j_{r-1} $. In this way, we can reformulate
$ H_{j_{1},\ldots , j_{r-1}} $ as
%
\begin{align}
\label{ThH2}
H_{j_{1},\ldots ,j_{r-1}} = \sum _{k=1}^{n} \psi _{e_{j}} (\nu _{j_{1},
\ldots , j_{r-1}}) \otimes \omega _{k},
\end{align}
according to in Equation \text{\eqref{Eq:hmr}}. By comparing Equations \text{\eqref{ThH1}} and \text{\eqref{ThH2}} we have
%
\begin{align}
\psi _{e_{j_{r}}}(\nu _{j_{1},\ldots , j_{r-1}}) &=\mathcal{M} (\phi _{e_{j_{1}}}
(\mu _{1}), \ldots , \phi _{e_{j_{r}}}(\mu _{r}))
\label{eq116}
\\
&= (e_{j_{1}} \otimes \cdots \otimes e_{j_{r}} ) *_{\phi }(\mu _{1},
\ldots , \mu _{r}).
\label{psidma}
\end{align}
The Weil pairing of rank $r$ Drinfeld modules over the domain
$\mathcal{A}$ is achieved from Equation \text{\eqref{Eq:h1hr}} by replacing
$ H_{j_{1},\ldots , j_{r-1}} $ with $ v_{j_{1},\ldots , j_{r-1}}$; namely
%
\begin{align}
\label{Thwp1}
\operatorname{Weil}_{I} (\mu _{1}, \ldots , \mu _{r} ) = \sum _{ 1
\leqslant j_{1},\ldots , j_{r-1}\leqslant n} \nu _{j_{1},\ldots ,j_{r-1}}
\otimes _{\mathcal{A}} (\omega _{j_{1}} \otimes \cdots \otimes
\omega _{j_{r-1}}).
\end{align}
Let $ \omega _{*} $ be the generator of
$ I^{-1}\Omega _{\mathcal{A}} / \Omega _{\mathcal{A}} $ and set
$ \omega _{j} = a_{j} \omega _{*} $ for some
$ a_{j} \in \mathcal{A}/ I $. Substituting
$ \omega _{j} = a_{j} \omega _{*} $ into Equation \text{\eqref{Thwp1}} gives
%
\begin{align}
\operatorname{Weil}_{I}(\mu _{1},\ldots ,\mu _{r}) & = \sum _{ 1
\leqslant j_{1},\ldots , j_{r-1}\leqslant n} \psi _{a_{j_{1}}\cdots a_{j_{r-1}}}(
\nu _{j_{1}, \ldots , j_{r-1}}) \otimes \omega _{*}^{\otimes (r-1)}
\nonumber
\\
& = \sum _{ 1\leqslant j_{1},\ldots , j_{r}\leqslant n} \Gamma _{j_{1},
\ldots ,j_{r}} \cdot \psi _{e_{j_{r}}}(\nu _{j_{1}, \ldots , j_{r-1}})
\otimes \omega _{*}^{\otimes (r-1)}
\label{Eq:Thwp2}
,
\end{align}
where we adopt the expansion
\begin{equation*}
a_{j_{1}} \cdots a_{j_{r-1}} = \sum _{j=1}^{n} \Gamma _{j_{1},\ldots ,j_{r-1},
j} e_{j},
\end{equation*}
for some $ \Gamma _{j_{1},\ldots , j_{r-1}, j} \in \mathbb{F}_{q}$. Let
us denote
\begin{equation*}
\operatorname{WO}_{I}^{(r)} = \sum _{ 1\leqslant j_{1},\ldots , j_{r}
\leqslant n} \Gamma _{j_{1},\ldots ,j_{r}} \cdot e_{j_{1}}\otimes e_{j_{2}}
\otimes \cdots \otimes e_{j_{r}}.
\end{equation*}
Then Equation \text{\eqref{Eq:Thwp2}} is equivalently expressed as
%
\begin{align}
\label{Eq:WeilDMa}
\operatorname{Weil}_{I}(\mu _{1},\ldots ,\mu _{r}) & =
\operatorname{WO}_{I}^{(r)} *_{\phi }(\mu _{1},\ldots , \mu _{r}),
\end{align}
by omitting the term $\omega _{*}^{\otimes (r-1)}$.
\end{proof}

\subsection{Weil operator over domain $\mathcal{A}$}
\label{sec5.3}

\begin{defn}
\label{defn5.10}
The symmetric form
$ \operatorname{WO}_{I}^{(r)} \in \operatorname{Sym}_{\mathbb{F}_{q}}^{r}
\mathcal{A}$ in \text{Theorem~\ref{Thm:wo}} is called the rank $r$ Weil operator
of the ideal $I $ with respect to the generator
$ \omega _{*} \in \Omega _{I} $. Note that
$ \operatorname{WO}_{I}^{(r)} $ does not depend on the Drinfeld module
$ \phi $.
\end{defn}
Suppose that $\mathcal{A}$ has two generators $x,y$, i.e.,
\begin{equation*}
\mathcal{A}\cong \mathbb{F}_{q}[x,y]/ (\rho (x,y)).
\end{equation*}
Then
\begin{equation*}
\operatorname{Sym}_{\mathbb{F}_{q}} ^ {r} \mathcal{A}\cong \mathbb{F}_{q}[X_{1},Y_{1},
\ldots ,X_{r},Y_{r}] / ( \rho (X_{i},Y_{i}) )_{i=1,\ldots ,r}.
\end{equation*}
As a consequence, $\operatorname{WO}_{I}^{(r)} $ can be expressed as a
pairwise symmetric polynomial in variables
$X_{1},Y_{1},\ldots ,X_{r},Y_{r}$. Let $\mathcal{B} $ be a basis of quotient
algebra $\mathcal{A}/ I$ with the dual basis of the form
$ \{ w(E) \cdot \omega _{*} | E \in \mathcal{B} \} $. According to the
proof of \text{Theorem~\ref{Thm:wo}}, the rank $r$ Weil operator of the ideal
$I $ with respect to the generator $ \omega _{*} \in \Omega _{I} $ can
be computed by the formula
%
\begin{align}
\label{Eq:Wo1}
\operatorname{WO}_{I}^{(r)} = \sum _{E_{1},E_{2},\ldots , E_{r-1}
\in \mathcal{B} } E_{1} \otimes E_{2} \otimes \cdots \otimes E_{r-1}
\otimes (\prod _{i=1}^{r-1}w_{k}(E_{i})).
\end{align}

In order to represent the Weil pairing of Drinfeld modules of rank two,
we consider two polynomials $\operatorname{O}_{I} '$ and
$\operatorname{O}_{I} '' $ in variables $X_{1}$, $X_{2}$, $Y_{1}$,
$Y_{2}$.

\begin{notation}%
\label{No:Wo}
Let $I$ be the principal ideal generated by $ P(x,y) $ with parameters
$ \alpha _{j}, ~\beta _{j},~\alpha _{d} $ as in Equation \text{\eqref{Eq:P}}.
For the case $ \alpha _{0} \not =0 $, we define the pairwise symmetric
polynomial
%
\begin{align}
\label{Eq:woiw}
\nonumber
\operatorname{O}_{I} ' =& \beta _{d-1}+ \mathrm{Tr}\zeta \alpha _{d} +
2 \zeta ^{q+1} \frac{\alpha _{d}}{\alpha _{0}} \beta _{0}+\sum _{k=0}^{d-1}
(\mathrm{Tr}\zeta + \frac{\beta _{0} \zeta ^{q+1}}{\alpha _{0} } )
\beta _{k} \Big(X_{1}^{d-k-1}Y_{1}+X_{2}^{d-k-1}Y_{2}\Big)
\\
\nonumber
&+\sum _{k=0}^{d-1} \big(\frac{\beta _{0} \zeta ^{q+1}}{\alpha _{0} }
\alpha _{k} + \beta _{k-1} - \zeta ^{q+1} \beta _{k}\big) \Big(X_{1}^{d-k}+X_{2}^{d-k}
\Big)
\\
&+Y_{1}Y_{2} \sum _{
\stackrel{j\geqslant 0, k\ge 0}{0\leqslant j+k\leqslant d-2}}
\frac{\beta _{k}}{2}\Big(X_{1}^{j} X_{2}^{d-2-j-k}+ X_{1}^{d-2-j-k} X_{2}^{j}
\Big)
\\
\nonumber
&+\Big(Y_{1}+Y_{2}\Big) \sum _{
\stackrel{j\geqslant 0, k\ge 0}{0\leqslant j+k\leqslant d-1}}
\frac{\alpha _{k}}{2}\Big(X_{1}^{j} X_{2}^{d-1-j-k}+ X_{1}^{d-1-j-k} X_{2}^{j}
\Big)
\\
\nonumber
&-\sum _{\stackrel{j\geqslant 1, k\ge 0}{1\leqslant j+k\leqslant d-1}}
\frac{\alpha _{k} \mathrm{Tr}\zeta -\beta _{k-1}+\zeta ^{k+1}\beta _{k}}{2}
\Big(X_{1}^{j} X_{2}^{d-j-k}+ X_{1}^{d-j-k} X_{2}^{j}\Big).
\end{align}
For $\beta _{0}\neq 0$, we define
%
\begin{align}
\label{Eq:woiv}
\nonumber
\operatorname{O}_{I}'' =& - \mathrm{Tr}\zeta \alpha _{d} + \beta _{d-1}
- 2 \frac{\alpha _{0}\alpha _{d}}{ \beta _{0}} +\sum _{k=0}^{d-1}
\big( -\frac{\alpha _{0}}{\beta _{0} } \beta _{k} + \alpha _{k}\big)
\Big(X_{1}^{d-k-1}Y_{1}+X_{2}^{d-k-1}Y_{2}\Big)
\\
\nonumber
&-\sum _{k=0}^{d-1} \Big((\mathrm{Tr}\zeta +
\frac{\alpha _{0}}{\beta _{0} }) \alpha _{k} - \beta _{k-1} +\zeta ^{q+1}
\beta _{k}\Big)\Big(X_{1}^{d-k}+X_{2}^{d-k}\Big)
\\
&+Y_{1}Y_{2} \sum _{
\stackrel{j\geqslant 0, k\ge 0}{0\leqslant j+k\leqslant d-2}}
\frac{\beta _{k}}{2}\Big(X_{1}^{j} X_{2}^{d-2-j-k}+ X_{1}^{d-2-j-k} X_{2}^{j}
\Big)
\\
\nonumber
&+\Big(Y_{1}+Y_{2}\Big) \sum _{
\stackrel{j\geqslant 0, k\ge 0}{0\leqslant j+k\leqslant d-1}}
\frac{\alpha _{k}}{2}\Big(X_{1}^{j} X_{2}^{d-1-j-k}+ X_{1}^{d-1-j-k} X_{2}^{j}
\Big)
\\
&-\sum _{\stackrel{j\geqslant 1, k\ge 0}{1\leqslant j+k\leqslant d-1}}
\frac{\alpha _{k} \mathrm{Tr}\zeta -\beta _{k-1}+\zeta ^{k+1}\beta _{k}}{2}
\Big(X_{1}^{j} X_{2}^{d-j-k}+ X_{1}^{d-j-k} X_{2}^{j}\Big).
\nonumber
\end{align}
\end{notation}
Notice that when both $ \alpha _{0} $ and $ \beta _{0} $ are non-zero,
the polynomials $ \operatorname{O}' $ and $ \operatorname{O}'' $ are related
by the relation
\begin{equation*}
\operatorname{O}_{I}'-\operatorname{O}_{I}'' = \kappa \cdot (P(X_{1},Y_{1})+P(X_{2},Y_{2})),
\end{equation*}
where
$\kappa =(\mathrm{Tr}\zeta +
\frac{\beta _{0} \zeta ^{q+1}}{\alpha _{0} } ) +
\frac{\alpha _{0}}{\beta _{0}} \in \mathbb{F}_{q}^{*}$. Therefore, both
polynomials are identical as elements in
$\operatorname{Sym}_{\mathbb{F}_{q}}^{2}\mathcal{A}$.

Applying the formula \text{\eqref{Eq:Wo1}}, we obtain the explicit formula of
the Weil pairing of rank two Drinfeld modules over $\mathcal{A}$.

\begin{thm}
\label{thm5.12}
Let $ \phi $ be a Drinfeld module of rank two over the base
$\mathcal{A}= \mathbb{F}_{q} [x,y ] $. Let $I$ be the principle ideal generated
by $ P(x,y) $. Then the Weil pairing of $ \phi $ can be written as
\begin{equation*}
\operatorname{Weil}_{I}(\mu _{1},\mu _{2}) = \operatorname{WO}_{I}^{(2)}
*_{\phi }(\mu _{1},\mu _{2}),
\end{equation*}
where
$ \operatorname{WO}_{I}^{(2)} \in \operatorname{Sym}^{2}_{\mathbb F_{q}}
\mathcal{A}$ (expressed as a polynomial in
$\mathbb{F}_{q}[X_{1},Y_{1},X_{2},Y_{2}]$) is given by
\begin{equation*}
\operatorname{WO}_{I}^{(2)} =
\begin{cases}
\operatorname{O}_{I}' & \text{ for $ \alpha _{0} \not = 0 $ }
\\
\operatorname{O}_{I}'' & \text{ for $ \beta _{0} \not = 0$},
\end{cases}
\end{equation*}
corresponding to the standard differential
$ \omega _{*}:= \frac{1}{P(x,y)} \omega _{\pi }$.
\end{thm}

\begin{proof}
We maintain the notation in Section~\ref{Sec:DB}. The computation is straightforward.
For the case $ \alpha _{0} \not = 0 $, we have the basis
$ \mathcal B_{1} $ of $ \mathcal{A}/ I $ and its dual
$ \{ w(E) \omega _{*} | E \in B_{1} \} $ with respect to the generator
$\omega _{*} $ by \text{Proposition~\ref{Prop:Thbd}}. Applying Equation \text{\eqref{Eq:Wo1}} for $ \mathcal B_{1} $ yields
%
\begin{align}
\label{Eq:CoWo1}
\begin{aligned}
\operatorname{WO}_{I}^{(r)} =&1\otimes w(1)+ x^{d-1}y\otimes w(x^{d-1}y)+
\sum _{j=1}^{d-1} x^{d-1-j}y\otimes w(x^{d-1-j}y)\\
&+\sum _{j=1}^{d-1} x^{d-j}
\otimes w(x^{d-j}).
\end{aligned}
\end{align}
Substituting $ w(1)$, $w(x^{d-1}y)$, $w(x^{d-1-j}y)$ and
$w(x^{d-j})$ (see Equations \text{\eqref{Eq:wxy}}-\text{\eqref{Eq:w1}}) into Equation \text{\eqref{Eq:CoWo1}}, we have
\begin{align*}
\label{Eq:CoWoiw}
\nonumber
\operatorname{WO}_{I}^{(r)} =& \sum _{k=1}^{d-1} \left (\! ( (
\mathrm{Tr}\zeta + \frac{\beta _{0} \zeta ^{q+1}}{\alpha _{0} } )
\beta _{k} + \alpha _{k}) X_{2}^{d-k-1}Y_{2} + (
\frac{\beta _{0} \zeta ^{q+1}}{\alpha _{0} } \alpha _{k} + \beta _{k-1}
{-} \zeta ^{q+1} \beta _{k}) X_{2}^{d-k}\! \right )
\\
\nonumber
& +\left ( \beta _{d-1} + \mathrm{Tr}\zeta \alpha _{d} + 2 \zeta ^{q+1}
\frac{\alpha _{d}}{\alpha _{0}} \beta _{0} \right ) + ( (\mathrm{Tr}
\zeta + \frac{\beta _{0} \zeta ^{q+1}}{\alpha _{0} } ) \beta _{0} +
\alpha _{0}) X_{2}^{d-1}Y_{2}
\\
&+
\frac {(\beta _{0}\zeta + \alpha _{0} )(\beta _{0} \zeta ^{q} + \alpha _{0} )}{\alpha _{0}}
X_{1}^{d-1}Y_{1}
\\
\nonumber
&+ \sum _{j=1}^{d-1} X_{1}^{d-1-j}Y_{1} \left (\sum _{k=0}^{j-1} (
\beta _{k} X_{2}^{j-k-1} Y_{2} + \alpha _{k} X_{2}^{j-k})+ \Big (
\alpha _{j}+ (\mathrm{Tr}\zeta +\frac{\beta _{0}}{\alpha _{0}}\zeta ^{q+1})
\beta _{j} \Big) \right )
\\
\nonumber
& + \sum _{j=1}^{d-1} X_{1}^{d-j} \Bigg( \sum _{k=0}^{j-1} ( (X_{2}^{j-1-k}Y_{2}-
\mathrm{Tr}\zeta X_{2}^{j-k}) \alpha _{k} + X_{2}^{j-k}\beta _{k-1} -X_{2}^{j-k}
\zeta ^{q+1}\beta _{k} )
\\
\nonumber
& + \Big ( \beta _{j-1}- \zeta ^{q+1}\beta _{j} +
\frac{\alpha _{j}}{\alpha _{0}} \zeta ^{q+1} \beta _{0} \Big) \Bigg),
\end{align*}
which coincides with $\operatorname{O}_{I}'$ given in \text{Notation~\ref{No:Wo}}.

The case $ \beta _{0} \not = 0 $ is analogous by applying the basis
$ \mathcal{B}_{2}^{*} $ and its dual in \text{Proposition~\ref{Prop:Thbd}}.
\end{proof}

Unlike the rank two case, the general Weil operator of rank
$r\geqslant 3$ is difficult to calculate and formulate. We demonstrate
how to compute the Weil operators of arbitrary rank
$ r \geqslant 2 $ for the cases $I=(x)$ and $I=(y)$. For
$ 1 \leqslant i \leqslant r $, we adopt the notations
\begin{equation*}
X_{i} = 1\otimes \cdots \otimes x\otimes \cdots \otimes 1,
\end{equation*}
where $x $ occurs in the $ i $-th position and $ Y_{i} $ similarly.

\subsubsection{Case: $I=(x)$}
\label{sec5.3.1}

Clearly, $\mathcal{B}_{1}= \{1, y\}$ is the basis of $ A/I$ with
$y^{2} = 0$. It follows from \text{Proposition~\ref{Prop:Thbd}} that
$\{ w(1)\omega _{*}, w(y) \omega _{*} \}$ is the dual basis, where
\begin{equation*}
\begin{cases}
w(1) = y
\\
w(y) = 1.
\end{cases}
\end{equation*}
For $r\geqslant 2$, from Equation \eqref{Eq:Wo1} the Weil operator of rank
$r$ is given by
\begin{align*}
\operatorname{WO}_{I}^{(r)} =&\sum _{e_{1},e_{2}, \cdots e_{r-1}\in
\mathcal{B}_{1}} e_{1} \otimes e_{2} \otimes \cdots \otimes e_{r-1}
\otimes \left (w(e_{1})w(e_{2})\cdots w(e_{r-1})\right )
\\
=& \sum _{i\geqslant 0}^{r-1} (w(y))^{i} (w(1))^{r-1-i}\sum _{1
\leqslant l_{1}<l_{2}\cdots < l_{i}\leqslant r-1}Y_{l_{1}}\cdots Y_{l_{i}}
\\
=& \sum _{1\leqslant l_{1}<l_{2}<\cdots <l_{r-2}\leqslant r-1}\Big( Y_{l_{1}}Y_{l_{2}}
\cdots Y_{l_{r-2}}Y_{r}\Big)+ Y_{1}Y_{2}\cdots Y_{{r-1}}
\\
=& \sum _{1\leqslant l_{1}<l_{2}<\cdots <l_{r-1}\leqslant r}\Big( Y_{l_{1}}Y_{l_{2}}
\cdots Y_{l_{r-1}}\Big).
\end{align*}

\subsubsection{Case: $I=(y)$}
\label{sec5.3.2}

Following the same procedure, we adopt the basis
$\mathcal{B}_{2}=\{1, x\}$ of $ \mathcal{A}/ I $ and its dual
$\{ v(1) \omega _{*}, v(x)\omega _{*}\}$, where
\begin{equation*}
\begin{cases}
v(1) = 1 - \zeta ^{q+1} x
\\
v(x) = -\zeta ^{q+1}.
\end{cases}
\end{equation*}
It follows from Equation \text{\eqref{Eq:rho}} that
\begin{equation*}
x^{2} = \zeta ^{-(q+1)} x.
\end{equation*}
Applying Equation \text{\eqref{Eq:Wo1}} again, we have
\begin{align*}
\operatorname{WO}_{I}^{(r)} =& 1-\zeta ^{(q+1)}\sum _{i=1}^{r}X_{i} +
\zeta ^{2(q+1)}\sum _{1\leqslant i<j\leqslant r}\Big( X_{i}X_{j}\Big)+
\cdots
\\
&+ (-1)^{k}\zeta ^{k(q+1)}\sum _{1\leqslant l_{1}<l_{2}<\cdots <l_{k}
\leqslant r}\Big( X_{l_{1}}X_{l_{2}}\cdots X_{l_{k}}\Big)
\\
&+\cdots +(-1)^{r-1}\zeta ^{(r-1)(q+1)}\sum _{1\leqslant l_{1}<l_{2}<
\cdots <l_{r-1}\leqslant r}\Big( X_{l_{1}}X_{l_{2}}\cdots X_{l_{r-1}}
\Big)
\\
=& \sum _{k=0}^{r-1}(-1)^{k}\zeta ^{k(q+1)}\sum _{1\leqslant l_{1}<l_{2}<
\cdots <l_{k}\leqslant r}\Big( X_{l_{1}}X_{l_{2}}\cdots X_{l_{k}}
\Big).
\end{align*}


%



%
%

\section*{Acknowledgements}
We would like to thank the editor and the anonymous reviewer for their valuable suggestions and detailed comments that we have used to improve the quality of our manuscript.

 \bibliographystyle{amsplain}
 \bibliography{paper}

\end{document}